\documentclass[reqno, 12pt]{amsart}
\usepackage[hmarginratio=1:1]{geometry}
\usepackage{amsaddr}
\usepackage{amsmath,amssymb,amsfonts,amsthm,mathtools,esint,mathrsfs,bm}
\usepackage{framed}
\usepackage{tikz}
\usepackage[font=small,labelfont=bf]{caption}
\usepackage{subcaption}
\usetikzlibrary{arrows}
\usetikzlibrary{matrix}
\usepackage{graphicx}
\usepackage{float}
\usepackage{subcaption}
\graphicspath{ {images/} }
\usepackage{color}
\usepackage{enumerate}

\newcommand{\vertiii}[1]{{\left\vert\kern-0.25ex\left\vert\kern-0.25ex\left\vert #1 
    \right\vert\kern-0.25ex\right\vert\kern-0.25ex\right\vert}}		
\newcommand{\e}{\varepsilon}		
\newcommand{\Hg}{\mathcal{H}_g}	
\newcommand{\w}{\rightharpoonup}		
\DeclareMathOperator{\dist}{dist}		
\DeclareMathOperator{\ind}{ind}		
\DeclareMathOperator{\Div}{div}		

\newtheorem{thm}{Theorem}[section]
\newtheorem{prop}[thm]{Proposition}
\newtheorem{lem}[thm]{Lemma}
\newtheorem{cor}[thm]{Corollary}
\theoremstyle{definition}

\numberwithin{equation}{section}

\title[On minimizers of the 2D Ginzburg-Landau energy with tangential anchoring]{On minimizers of the 2D Ginzburg-Landau energy with tangential anchoring}
\author[Alama, Bronsard and van Brussel]{Stan Alama, Lia Bronsard and Lee van Brussel}
\address{Department of Mathematics and Statistics, McMaster University, Hamilton, ON, Canada}
\email{alama@mcmaster.ca, bronsard@mcmaster.ca, vanbrulw@mcmaster.ca}
\date{\today}

\begin{document}
\maketitle


\begin{abstract}
We analyze Ginzburg--Landau minimization problems in two dimensions with either a ``strong or weak" tangential boundary condition. These problems are motivated by experiments in liquid crystal with boundary defects. In the singular limit when the correlation length tends to zero, we show that boundary defects will be observed for weak anchoring, while both boundary and interior vortices are possible for strong anchoring in the first order limit. 
\end{abstract}


\section{Introduction}\label{sec:intro}
In this paper we study minimizers of two variational problems motivated by the study of defects in a nematic liquid crystal.  We consider a two-dimensional setting, related to a thin-film reduction of the three dimensional Landau--de Gennes model to two dimensions.  The special feature we are interested in comes from the work of Volovik and Lavrentovich \cite{volovik1983topological} where nematic drops are placed in an isotropic medium, allowing for the control of nematic boundary behaviour. In this way, the liquid crystal and its associated defect dynamics are studied as the nematic boundary molecules are transformed from having a forced angle of $\alpha=\pi/2$ with respect to the unit normal $n$ to the boundary of the droplet to $\alpha=0$. In this paper, we return to the well-studied Ginzburg-Landau functional but with new tangential types of boundary conditions, inspired by this physical phenomena.

We begin by describing the variational problem in mathematical terms, and stating our main results in Theorems~\ref{thm:big1} and \ref{thm:big2}. We consider a two-dimensional, bounded, simply connected domain $\Omega\subset\mathbb{R}^2\cong\mathbb{C}$ representing the space occupied by the liquid crystal with $C^{3,\alpha}$-smooth boundary $\Gamma:=\partial\Omega$. Let $g:\Gamma\to\mathbb{S}^1$ be $C^{3,\alpha}$-smooth boundary data with positive degree
\begin{equation*}
\mathcal{D}:=\deg(g;\Gamma)>0.
\end{equation*}
A natural example is to choose $g$ to parametrize the (positively oriented) unit tangent vector to $\partial \Omega$, but this need not be the case.
In order to force the order parameter $u$ to be parallel (or close to parallel) with respect to $g$, we will be using two methods. The first method is to impose that $u$ have zero projection along a vector orthogonal to $g$, that is, impose the pointwise scalar product condition $\langle u,g^{\perp}\rangle=0$ on $\Gamma$. With this, we consider the Ginzburg--Landau energy defined for $H^1(\Omega;\mathbb{R}^2)$ mappings
\begin{equation*}
E_{\e}(u):=\frac{1}{2}\int_{\Omega}\left(|\nabla u|^2+\frac{1}{2\e^2}\left(1-|u|^2\right)^2\right)dx
\end{equation*}
where $\e>0$ and observe the behaviour of solutions to the \emph{strong tangential minimization problem}
\begin{equation}\label{minprobstrong}
\inf\left\{E_{\e}(u):u\in \Hg(\Omega):=\{u\in H^1(\Omega;\mathbb{R}^2):\langle u,g^{\perp}\rangle=0\ \mbox{on}\ \Gamma\}\right\},
\end{equation}
in the limit as $\e\to 0$. In this way, we can observe the topological defects associated to the limiting map by analyzing a sequence of energy minimizing configurations $\{u_{\e}\}$ where the nematic material is asked to be precisely (strongly) parallel to $g$ along $\Gamma$ for each $\e>0$. It turns out that asking such a condition to hold does not quite translate to a standard Dirichlet or Neumann problem for the associated Euler--Lagrange equations, but rather a mixture of the two within appropriate coordinates. To see this, we make the additional assumption that $g$ be defined on a tubular neighborhood 
\begin{equation*}
\mathcal{N}_{\Gamma}:=\{x\in\overline{\Omega}:\dist(x,\Gamma)<\delta\}
\end{equation*}
with $\delta>0$ is small. Using this assumption, there exists a natural decomposition for functions $u$ in $\mathcal{N}_{\Gamma}$ using the orthonormal frame $\{g(x),g^{\perp}(x)\}$ via
\begin{equation}\label{def:udecomp}
u=u_{\parallel}g+u_{\perp}g^{\perp}
\end{equation}
where $u_{\parallel}:=\langle u,g\rangle$ and $u_{\perp}:=\langle u,g^{\perp}\rangle$. In using this decomposition, we find that solutions to the strong tangential problem \eqref{minprobstrong} satisfy the Euler--Lagrange system
\begin{equation}\label{eq:ELstrong}
\left\{
\begin{alignedat}{2}
-\Delta u&=\frac{1}{\e^2}(1-|u|^2) u\quad && \mbox{in}\ \Omega,\\
u_{\perp}&=0\quad && \mbox{on}\ \Gamma,\\
\partial_nu_{\parallel}&=0\quad && \mbox{on}\ \Gamma.
\end{alignedat}
\right.
\end{equation}
The second method for enforcing parallelity is done through boundary energy penalization (see e.g. Moser \cite{moser2003}). Indeed, define
\begin{equation*}
E_{\e}^{g,s}(u):=E_{\e}(u)+\frac{1}{2\e^s}\int_{\Gamma}\langle u,g^{\perp}\rangle^2\,ds
\end{equation*}
where $s\in (0,1]$, so that solutions of the \emph{weak tangential minimization problem} 
\begin{equation}\label{minprobweak}
\inf\left\{E_{\e}^{g,s}(u):u\in H^1(\Omega;\mathbb{R}^2)\right\}
\end{equation}
are energetically induced to decrease their projection along $g^{\perp}$. By calculating the first variation for $E_{\e}^{g,s}$, it can be easily shown that minimizers $u_{\e}$ satisfy the weak anchoring system
\begin{equation}\label{eq:ELweak}
\left\{
\begin{alignedat}{2}
-\Delta u&=\frac{1}{\e^2}(1-|u|^2) u\quad && \mbox{in}\ \Omega,\\
\partial_nu&=-\frac{1}{\e^s}u_{\perp}g^{\perp}\quad && \mbox{on}\ \Gamma.
\end{alignedat}
\right.
\end{equation}
For either minimization problem, solutions to \eqref{minprobstrong} and \eqref{minprobweak} are guaranteed by the direct method from the calculus of variations. Moreover, it can also be shown that strong tangential minimizers, in some sense, are weak limits of solutions to a certain modified weakly tangential minimization problem. Therefore, both problems are naturally connected and it is reasonable to analyze the solutions of both.\\[0.5em]
It is well known from the literature (see Bethuel-Brezis-H\'elein \cite{bbh}, for example) that the local winding behaviour of minimizers about vortices and the global winding behaviour of the boundary data $g$ are directly linked to the energy of a minimizing configuration. Thus, given that our interest is in the observation of boundary defects, we must grasp, in some way, the winding behaviour of minimizers near boundary vortices. Indeed, when a defect is located in the interior, this winding is easily quantifiable by calculating the standard topological degree of the minimizer's normalization about a small circle centered at the defect. However, since a closed curve cannot be made about a boundary vortex, it is not immediately clear how one should proceed in this case. To combat this, we develop a topological quantity called the \emph{boundary index}, which essentially counts the net number of approximate $\pi$-rotations that are made from one side of the vortex to the other on the boundary. In this way, a boundary defect with an associated boundary index $d$ will resemble an interior vortex of degree $d$ cut in half, and thus carry a ``half-integer" degree (see Definition \eqref{def:bdryindex}). A rigorous construction of the boundary index is given in Section \ref{sec:wind}. Using the notion of the boundary index, the main results of this paper are summarized in Theorems \ref{thm:big1} and \ref{thm:big2}.
\begin{thm}\label{thm:big1}
Suppose $\{u_{\e}\}_{\e>0}$ is a sequence of solutions to \eqref{minprobstrong} with associated boundary function $g:\mathcal{N}_{\Gamma}\to\mathbb{S}^1$ of degree $\mathcal{D}=\deg(g;\Gamma)\geq 1$. Then there is a subsequence $\e_n\to 0$, a finite number of point singularities $\Sigma\subset\overline{\Omega}$ and a harmonic map $u_0\in H^1(\overline{\Omega}\setminus\Sigma;\mathbb{S}^1)$ such that
\begin{equation*}
u_{\e_n}\w u_0\quad \mbox{weakly in}\ H^1_{loc}(\overline{\Omega}\setminus\Sigma;\mathbb{R}^2)
\end{equation*}
with each defect contained in $\Sigma$ having either associated degree, or boundary index, equal to one. In the particular case where $\mathcal{D}=1$, then one and only one of the following scenarios hold:
\begin{enumerate}
\item$\Sigma=\{p\}$ with $p\in\Omega$,\\
\item $\Sigma=\{q_1,q_2\}$ with $q_1,q_2\in\Gamma$.
\end{enumerate}
\end{thm}
For the last part of Theorem \ref{thm:big1}, we remind the reader that our primary motivation for studying this problem came from the topological observations made on 3D samples, by Volovik and Lavrentovich in \cite{volovik1983topological}. In particular, they found experimentally single interior hedgehog defect when molecules are asked to be normal to the boundary and a bipolar boojum pair when requiring tangential conditions. In either case, the normal and tangential boundary data are of degree one and thus our theoretical treatment of the problem weakly recovers this observation. To get a complete picture, a renormalized energy analysis would need to be conducted in order to show when one defect type is preferred over another. To this end, in the last section of this work, we provide a concrete example of strong tangential anchoring in the case $\Omega=B_1(0)$, the unit disc, with $g=\tau$ the positively oriented unit tangent vector to the boundary $\Gamma$,  to highlight that the boundary vortex pair may give the preferable energy minimizing configuration. Such a result in 2D would be a first step in obtaining theoretically results coinciding with the found experimental data.\\

Next we state our result in case of weak tangential boundary conditions:

\begin{thm}\label{thm:big2}
Suppose $\{u_{\e}\}_{\e>0}$ is a sequence of solutions to \eqref{minprobweak} with associated boundary function $g:\mathcal{N}_{\Gamma}\to\mathbb{S}^1$ of degree $\mathcal{D}=\deg(g;\Gamma)\geq 1$. Then there is a subsequence $\e_n\to 0$, a finite number of point singularities $\Sigma\subset\overline{\Omega}$ and a harmonic map $u_0\in H^1(\overline{\Omega}\setminus\Sigma;\mathbb{S}^1)$ such that
\begin{equation*}
u_{\e_n}\w u_0\quad \mbox{weakly in}\ H^1_{loc}(\overline{\Omega}\setminus\Sigma;\mathbb{R}^2)
\end{equation*}
with each defect contained in $\Sigma$ having associated degree or boundary index equal to one. If $s\in(0,1)$, it holds that $\Sigma\subset\Gamma$ with $|\Sigma|=2\mathcal{D}$.
\end{thm}
The primary takeaway of Theorem \ref{thm:big2} comes from the observation that the exponent $s\in(0,1]$ \emph{almost} completely dictates the allocation of defects in $\overline{\Omega}$. In particular, the Theorem states that independent of the winding behaviour of $g$ and the geometry of $\Omega$, giving vortices `more room' along the boundary (on the scale of $\e^s$ as opposed to $\e$) is enough for boundary vortex pairs to \emph{always} be energetically preferable when compared to interior vortices.

%

A related model is that of a thin ferromagnetic film as obtained in an appropriate limiting regime by DeSimone, Kohn, Muller and Otto \cite{desimone2002reduced}. This limiting ferromagnetic thin film was studied by Moser \cite{moser2003}, and by Kurzke \cite{kurzke2006boundary} in certain settings. In those problems, they impose  tangential weak anchoring conditions with $g=\tau$ the unit tangent, and find critical anchoring strength at which boundary vortices are favored over interior vortices. In our case, we also consider a strong tangential anchoring, and generalize their results to weak anchoring for any $g$.  More recently, Ignat-Kurzke \cite{IK} have obtained $\Gamma$-convergence results for the weak tangential anchoring problem using a notion of global Jacobian in a different limiting regime where interior vortices cost more energy than boundary vortices. In the context of polarization-modulated orthogonal smectic liquid crystal, Garcia-Cervera, Giorgi and Joo \cite{GCGJ} have studied boundary vortices in a square domain with mixed weak and strong boundary conditions on the edges.

In the work of Volovik and Lavrentovich \cite{volovik1983topological}  the topological dynamics of the nematic material are observed as the boundary molecules are changed from being parallel to the boundary to perpendicular, by varying the aperture  $\alpha$ of the cone formed by the molecule with the normal to the boundary. When the angle $\alpha=\pi/2$, a bipolar structure is noticed with two point defects occurring along the boundary called \emph{boojums}. At the other extreme with $\alpha=0$, a single interior hedgehog defect is realized. Using a vector-valued order parameter $u$ for modelling the molecular alignment of the liquid crystal, the authors note that a surface energy density proportional to
\begin{equation}\label{eq:VLcond}
[\langle u,n\rangle^2-\cos^2\alpha]^2
\end{equation}
can be used in comparison to an interior gradient energy for determining the energy preference of boojum defects to hedgehog defects and vice versa. A generalization of this setup was analyzed  by Alama, Bronsard and Golovaty in \cite{alama2020thin} where they replaced the boundary's normal vector $n$ in expression \eqref{eq:VLcond} with general smooth $\mathbb{S}^1$-valued boundary data $g$, possessing a positive associated winding number along the boundary, and restricting $\alpha\in(0,\pi/2)$, the relative angle made between $g$ and the order parameter $u$. In this present work, we aim to answer the question of how this generalization operates for the specific case of $\alpha=\pi/2$ using Ginzburg-Landau as a toy model for nematic material. In particular, we are interested in obtaining conditions for which minimizing configurations prefer boundary defects over interior defects in this setting.\\[0.5em]


The rest of the paper is organized as follows: in Section 2 we present  upper bounds for the energy of minimizers to each problem, as well as {\it a priori} pointwise bounds for all solutions of the associated Euler-Lagrange equations, adapted for our new settings. In Section 3, we present our $\eta$ compactness results adapted from Struwe \cite{struwe1994asymptotic} to handle each type of boundary conditions and use it to define the ``bad balls" for each type, and show that they are contained in a finite number of very small balls. Next in Section 4, we analyze the winding behaviour of minimizers around boundary defects and introduce our notion of boundary index and use it to obtain the important ``degree Proposition and Lemma" (Proposition~\ref{prop:winding} and Lemma~\ref{lem:addthemup}) which will be essential in proving the lower bound on the energy of boundary defects in terms of the degree of the boundary data.  In Section 5, we obtain an energy lower bound for each type of tangential conditions over the appropriate ball collections, and we put everything together and prove our two main theorems. Finally in Section 6, we present an example with strong tangential anchoring on the unit disc which suggest that two antipodal boundary defect would be favored.


\section{Preliminary Facts for Minimizers}\label{sec:prelim}
We begin by showing an important pointwise bound on solutions of the Euler--Lagrange systems  \eqref{eq:ELstrong} and \eqref{eq:ELweak}. The proof follows familiar lines, (see \cite{bbh, alama2015weak, alama2020thin}) and so we provide a sketch to highlight the differences with previous papers. 

\begin{lem}\label{lem:bounded}
Suppose $u$ is a solution of \eqref{eq:ELstrong} or \eqref{eq:ELweak}. Then $|u|\leq 1$ and there is a constant $C_0>0$ independent of $\e$ for which $\e|\nabla u|\leq C_0$ for all $x\in\Omega$.
\end{lem}
\begin{proof}
Define $V:=|u|^2-1$ and $V_+:=\max\{V,0\}$. Using the Euler--Lagrange equations and integrating by parts over $\Omega$ we obtain
\begin{equation*}\label{ineq:pwineq}
0\leq \int_{\Omega} |u|^2 V_+^2\,dx\leq \frac{1}{2}\int_{\Gamma}V_+\partial_n V\,ds-\frac{1}{2}\int_{\Omega}|\nabla V_+|^2\,dx.
\end{equation*}
If $u$ is a solution of \eqref{eq:ELstrong}, it follows that $V_+\equiv 0$. If $u$ is a solution of \eqref{eq:ELweak}, then
\begin{equation*}
\partial_nV=2\langle u,\partial_n u\rangle=-\frac{2}{\e^s}(u_{\perp})^2\leq 0
\end{equation*}
and we obtain $V_+\equiv 0$ again. Thus, $|u|\leq 1$ in $\Omega$.\\

The gradient bound can be obtained by contradiction: suppose that there exists sequences $\e_k\to 0$ and $x_k\in\overline{\Omega}$ so that $t_k:=|\nabla u_k(x_k)|=\|\nabla u_k\|_{\infty}$ satisfies $t_k\e_k\to \infty$ as $k\to\infty$. Let
\begin{equation*}
v_k(x):=u_k\left(x_k+\frac{x}{t_k}\right)
\end{equation*}
which is defined whenever $y:=x_k+x/t_k\in\Omega$. Likewise, define $h(x):=g(y)$ whenever $y\in\Gamma$. By the uniform bound on $u$ proven above and the choice of scaling, we have $\|v_k\|_{\infty}\leq 1$ and
\begin{equation}\label{eq:gradvbound}
|\nabla v_k(0)|=1\ \forall k.
\end{equation}
By the uniform bound $\| v_k\|_{\infty}\leq 1$ and using the fact that $v_k$ solves
\begin{equation*}
-\Delta v_k=\frac{1}{(t_k\e_k)^2}(1-|v_k|^2)v_k\quad\ \mbox{for}\ x\in t_k[\Omega-x_k],
\end{equation*}
we conclude $\Delta v_k\to 0$ uniformly on $\Omega$. There are two blow-up cases to consider. If along some subsequence $t_k\dist(x_k,\Gamma)\to +\infty$, then $v_k\to v$ with $v$ bounded and harmonic in all of $\mathbb{R}^2$, and hence constant, contradicting \eqref{eq:gradvbound}.\\

Suppose now $t_k\dist(x_k,\Gamma)$ is bounded uniformly so that the domains of $v_k$ converge to the half-space
\begin{equation*}
t_k[\Omega-x_k]\to\mathbb{R}^2_+\ \mbox{as}\ k\to\infty.
\end{equation*}
For each $k$, the weak tangential problem becomes
\begin{equation}\label{eq:weakrescale}
\left\{
\begin{alignedat}{2}
-\Delta v_k&=\frac{1}{(t_k\e_k)^2}(1-|v_k|^2) v_k\quad && \mbox{in}\ t_k[\Omega-x_k],\\[0.5em]
\partial_nv_k&=-\frac{1}{t_k\e_k^s}\langle v_k,h^{\perp}\rangle h^{\perp}\quad && \mbox{on}\ t_k[\Gamma-x_k],
\end{alignedat}
\right.
\end{equation}
while the strong tangential problem is written
\begin{equation}\label{eq:strongrescale}
\left\{
\begin{alignedat}{2}
-\Delta v_k&=\frac{1}{(t_k\e_k)^2}(1-|v_k|^2) v_k\quad && \mbox{in}\ t_k[\Omega-x_k],\\[0.5em]
\langle v_k,h^{\perp}\rangle&=0\quad && \mbox{on}\ t_k[\Gamma-x_k],\\[0.5em]
\partial_n\langle v_k,h\rangle&=0\quad && \mbox{on}\ t_k[\Gamma-x_k].
\end{alignedat}
\right.
\end{equation}
Both problems yield a bounded harmonic limit $v$ defined on $\mathbb{R}^2_+$ with $v_k\to v$ in $C^k_{loc}$. Since $v_k$ and $h$ are bounded uniformly, the normal derivative of system \eqref{eq:weakrescale} has the limit $\partial_n v_k\to 0$ as $k\to\infty$ and so the limiting harmonic map $v$ satisfies the Neumann condition $\partial_n v=0$ on $\partial\mathbb{R}^2_+$. By the reflection principle, there exists a bounded harmonic extension of $v$ to all of $\mathbb{R}^2$ which again contradicts \eqref{eq:gradvbound} by Liouville's theorem. \\

For system \eqref{eq:strongrescale}, the boundary data $h$ converges to a constant vector field on $\partial\mathbb{R}^2_+$ and the boundary conditions imply $\langle v,h^{\perp}\rangle=\partial_n \langle v,h\rangle=0$ along $\partial\mathbb{R}^2_+$. Let $\tilde{h}$ denote the extension of $h$ to $\overline{\mathbb{R}^2_+}$ and note that $\langle v,\tilde{h}\rangle$ is a harmonic scalar function defined on $\mathbb{R}^2_+$. By the reflection principle and Liouville's theorem, $\langle v,\tilde{h}\rangle$ extends to a constant function on $\mathbb{R}^2$. Next, since $\langle v,h^{\perp}\rangle=0$ on $\partial\mathbb{R}^2_+$ and $\langle v,\tilde{h}\rangle$ is constant, it must be that $\langle v,\tilde{h}\rangle=\pm|v|$ on all of $\mathbb{R}^2$ and thus $v$ is constant. Condition \eqref{eq:gradvbound} is contradicted once again giving $\e|\nabla u|\leq C_0$ for all $x\in\Omega$ where $C_0$ is a constant independent of $\e$.
\end{proof}


\begin{prop}\label{prop:upperbound}
If $u_{\e}$ is a strongly tangential minimizer for $E_{\e}$, then
\begin{equation}\label{ineq:upper2}
E_{\e}(u_{\e})\leq \pi \mathcal{D}|\ln \e|+C
\end{equation}
with $C>0$ a constant independent of $\e$. If $u_{\e}$ is a weakly tangential minimizer for $E_{\e}^{g,s}$, then there is a constant $C>0$ independent of $\e$ so that
\begin{equation}\label{ineq:upper}
E_{\e}^{g,s}(u_{\e})\leq \pi s\mathcal{D}|\ln \e|+C.
\end{equation}
\end{prop}


The proof of this proposition utilizes a local polar coordinate system near the boundary which is defined in \cite{alama2015weak, alama2020thin, kurzke2006boundary} and which we will use throughout the paper. For convenience, we provide a brief description here. Let $x_0\in\overline{\Omega}$ and $R>0$. Set
\begin{equation*}
\omega_R(x_0):=B_R(x_0)\cap\Omega
\end{equation*}
and in the case where $x_0\in\Gamma$, define
\begin{equation*}
\Gamma_R(x_0):=\overline{\omega_R(x_0)}\cap\Gamma.
\end{equation*}
Whenever $x_0\in\Gamma$, $\tau(x_0)$ will denote the positively oriented unit tangent vector to $\Gamma$ at $x_0$. Using $\tau(x_0)$ as a reference, the polar coordinates $(r,\theta)$ centered at $x_0$ can be defined so that $\theta$ is the angle measured from the ray defined by $\tau(x_0)$ and $r=|x-x_0|$. By the smoothness of $\Gamma$, note that $R$ can be chosen small enough so that 
\begin{align*}
\omega_R(x_0)=\left\{(r,\theta):\theta_1(r)<\theta<\theta_2(r),\ 0<r<R\right\}
\end{align*}
where $\theta_1(r)$ and $\theta_2(r)$ are smooth functions satisfying
\begin{align}\label{bounds:theta}
|\theta_1(r)|,\ |\pi-\theta_2(r)|\leq cr
\end{align}
for some constant $c=c(\Gamma)\geq 0$. These coordinates allow us to parametrize $\Gamma_R(x_0)\setminus\{x_0\}$ in two pieces:
\begin{equation*}
\Gamma_R^+(x_0):=\{(r,\theta_1(r)):0<r<R\},\quad \Gamma_R^-(x_0):=\{(r,\theta_2(r)):0<r<R\}.
\end{equation*}
Annular regions are defined similarly. For any $x_0\in\overline{\Omega}$ set
\begin{equation*}
A_{r_1,r_2}(x_0):=\omega_{r_2}(x_0)\setminus\overline{\omega_{r_1}(x_0)},\quad 0<r_1<r_2.
\end{equation*}
When $x_0\in\Gamma$ and $r_2>0$ is taken small enough, the intersection $\overline{A_{r_1,r_2}(x_0)}\cap \Gamma$ consists of two disjoint smooth arcs
\begin{equation}\label{eq:paramsgammas}
\begin{split}
\Gamma_{r_1,r_2}^+(x_0)&=\{(r,\theta_1(r)):r_1<r<r_2\},\\
\Gamma_{r_1,r_2}^-(x_0)&=\{(r,\theta_2(r)):r_1<r<r_2\},
\end{split}
\end{equation}
where $\theta_1(r)$ and $\theta_2(r)$ are as in \eqref{bounds:theta}. For notational convenience, we also set
\begin{equation*}
\Gamma_{r_1,r_2}^{\pm}(x_0):=\overline{A_{r_1,r_2}(x_0)}\cap\Gamma=\Gamma_{r_1,r_2}^+(x_0)\cup \Gamma_{r_1,r_2}^-(x_0).
\end{equation*}
Lastly, we define a localized energy on subsets $\omega_r(x_0)$ by
\begin{align*}
E_{\e}(u;\omega_r)&:=\frac{1}{2}\int_{\omega_r}\left(|\nabla u|^2+\frac{1}{2\e^2}(1-|u|^2)^2\right)dx,\\
E_{\e}^{g,s}(u;\omega_r)&:=E_{\e}(u;\omega_r)+\frac{1}{2\e^s}\int_{\Gamma\cap\overline{\omega_r}}\langle u,g^{\perp}\rangle^2\,ds.
\end{align*}


\begin{proof}[Proof of Proposition \ref{prop:upperbound}]\ \\[0.5em]
For inequality \eqref{ineq:upper2}, the desired bound is a consequence of \cite[Lemma 2.1]{struwe1994asymptotic}. Let $v_{\e}$ be a minimizer for $E_{\e}$ over $H^1_{g}(\Omega)=\{v\in H^1(\Omega;\mathbb{R}^2):v=g\ \mbox{on}\ \Gamma\}$. The inclusion $H^1_{g}(\Omega)\subset \mathcal{H}_g(\Omega)$ implies  $E_{\e}(u_{\e})\leq E_{\e}(v_{\e})$ and applying \cite[Lemma 2.1]{struwe1994asymptotic} to $E_{\e}(v_{\e})$ yields
\begin{equation*}
E_{\e}(u_{\e})\leq E_{\e}(v_{\e})\leq \pi\mathcal{D}|\ln\e|+C.
\end{equation*}
For weakly tangential minimizers, a test function is constructed following \cite[Lemma 3.1]{alama2015weak} and \cite[Proposition 3.1]{kurzke2006boundary}. Consider $2\mathcal{D}$ sets of the form $\omega_R(q_j)$ where $\{q_j\}_{j=1}^{2\mathcal{D}}$ are well-separated points on $\Gamma$ and $R$ is chosen so that 
\begin{equation*}
2\e^s<R<\frac{1}{2}|q_i-q_i| 
\end{equation*}
for all indices $i\neq j$. Assume that the points $\{q_j\}_{j=1}^{2\mathcal{D}}$ are labeled such that $q_{j+1}$ is the first point found by following the positively oriented tangent vector field along $\Gamma$ starting from $q_{j}$. These points partition $\Gamma$ into $2\mathcal{D}$ smooth segments $C_j$ in the sense that $\Gamma=\cup_{j=1}^{2\mathcal{D}}C_j$ with $C_j$ being the curve connecting $q_j$ and $q_{j+1}$. Let $\gamma$ be a lifting of $g$ on the curve $\Gamma_R(q_j)$, that is, $g=e^{i\gamma}$ on $\Gamma_R(q_j)$, and define
\begin{align*}
h_1(r)&=\gamma\left(re^{i\theta_1(r)}\right)+(j-1)\pi,\\[0.5em]
h_2(r)&=\gamma\left(re^{i\theta_2(r)}\right)+j\pi,\\[0.5em]
\phi(r,\theta)&=\frac{h_2(r)-h_1(r)}{\theta_2(r)-\theta_1(r)}(\theta-\theta_1(r))+h_1(r),
\end{align*}
where $\theta_1(r)$ and $\theta_2(r)$ are as in \eqref{bounds:theta}. In this way we have $e^{i\phi(r,\theta)}=g$ on $\Gamma_R^+(q_j)$ for $j$ odd and on $\Gamma_R^-(q_j)$ for $j$ even. Similarly, we get $e^{i\phi(r,\theta)}=-g$ for the opposite parities.  Next, let $\eta_{\e}(r)\in C^{\infty}$ be a cut-off function near $q_j$ satisfying $0\leq \eta_{\e}\leq 1$ for all $r$, $\eta_{\e}(r)=1$ for $r\geq 2\e^s$ and $\eta_{\e}(r)=0$ for $r<\e^s$. In setting
\begin{equation*}
\psi_j(r,\theta):=\eta_{\e}(r)\phi(r,\theta)+(1-\eta_{\e}(r))(\gamma(q_j)+(j-1)\pi)
\end{equation*}
we may define the $\mathbb{S}^1$-valued test function $v_{\e}^{(j)}=e^{i\psi_j(r,\theta)}$ on $\omega_{R}(q_j)$, which by construction, simulates a half-vortex in the annular region $A_{2\e^s,R}(q_j)$. Using the properties of the cut-off function along with Cauchy-Schwarz and the fact that $v_{\e}^{(j)}$ is $\mathbb{S}^1$-valued,
\begin{align*}
\frac{1}{2\e^2}\int_{\omega_R(q_j)}(1-|v_{\e}|^2)^2\,dx=0\quad \mbox{and}\quad \frac{1}{2\e^s}\int_{\Gamma_R(q_j)}\langle v_{\e},g^{\perp}\rangle^2\,ds\leq C
\end{align*}
where $C>0$ is a constant independent of $\e$. 
The Dirichlet energy of $v_{\e}^{(j)}$ on $\omega_R(q_j)$, can be conveniently estimated using polar coordinates: it is a straightforward calculation to confirm that the radial component $\int_{\omega_{R}(q_j)}|\partial_{r}v_{\e}|^2\ dx$ is uniformly bounded in $\e$. To bound the angular energy over $\omega_{2\e^s}(q_j)$, we note that by \eqref{bounds:theta} and the smoothness of $\gamma$, there is a constant $c>0$ so that
\begin{align*}
|h_2(r)-h_1(r)|\leq \pi+cr,\quad \mbox{and}\quad |\theta_2(r)-\theta_1(r)|\geq \pi-cr.
\end{align*}
Thus, we have:
\begin{equation*}
\int_{\omega_{2\e^s}(q_j)}\frac{1}{r^2}|\partial_{\theta}v|^2\,dx\leq\frac{(\pi+cR)^2}{(\pi-cR)}\ln 2<\infty.
\end{equation*}
Therefore it must be that the primary energy contribution comes from the subset $A_{2\e^s,R}(q_j)$. Using the same estimates as above,
\begin{equation*}
\int_{\omega_{R}(q_j)}\frac{1}{r^2}|\partial_{\theta}v_{\e}|^2\,dx\leq \pi s|\ln\e|+C
\end{equation*}
for $C>0$ independent of $\e$ and so
\begin{equation*}
E_{\e}^{g,s}(v_{\e};\omega_R(q_j))\leq \frac{\pi}{2} s|\ln\e|+C
\end{equation*}
on each $\omega_R(q_j)$. Finally, we must connect these local test functions in a way that is independent of $\e$. Consider the punctured domain
\begin{equation*}
\tilde{\Omega}:=\Omega\setminus\bigcup_{j=1}^{2\mathcal{D}}\omega_R(q_j)
\end{equation*}
with boundary given by
\begin{equation*}
\tilde{\Gamma}:=\partial\tilde{\Omega}=\left(\Gamma\setminus \cup_{j=1}^{2\mathcal{D}}\Gamma_R(q_j)\right)\bigcup\left(\cup_{j=1}^{2\mathcal{D}}\partial B_R(q_j)\cap\Omega\right).
\end{equation*}
We set the orientation of $\tilde{\Gamma}$ to match that of $\Gamma$ where they coincide. With this orientation, the function $\tilde{g}:\tilde{\Gamma}\to\mathbb{S}^1$ defined by
\begin{align*}
\tilde{g}:=\left\{
\begin{array}{cl}
g & \mbox{on}\ \tilde{\Gamma}\cap C_j\ \mbox{for}\ j\ \mbox{odd}\\[0.5em]
-g & \mbox{on}\ \tilde{\Gamma}\cap C_j\ \mbox{for}\ j\ \mbox{even}\\[0.5em]
v_{\e}^{(j)} & \mbox{on}\ \partial B_R(q_j)\cap\Omega\ \mbox{for each}\ j=1,\ldots,2\mathcal{D},
\end{array}
\right.
\end{align*}
satisfies $\deg(\tilde{g};\tilde{\Gamma})=0$ by construction and therefore we may let $V$ be the $\mathbb{S}^1$-valued harmonic extension of $\tilde{g}$ to $\tilde{\Omega}$ which has uniformly bounded Dirichlet energy. Setting 
\begin{align*}
H_{\e}=\left\{
\begin{array}{ll}
V & \mbox{in}\ \tilde{\Omega},\\[0.5em]
v_{\e}^{(j)} & \mbox{in}\ \omega_{R}(q_j)\ \mbox{for each}\ j=1,\ldots,2\mathcal{D},
\end{array}
\right.
\end{align*}
we obtain a bound on $E_{\e}^{g,s}(u_{\e})$ via
\begin{equation*}
E_{\e}^{g,s}(u_{\e})\leq E_{\e}^{g,s}(H_{\e})=\sum_{j=1}^{2\mathcal{D}}E_{\e}^{g,s}(v_{\e}^{(j)};\omega_R(q_j))+E_{\e}^{g,s}(V;\tilde{\Omega})\leq \pi s\mathcal{D}|\ln\e|+C
\end{equation*}
as desired.
\end{proof}


\section{$\eta$-Compactness and Related Consequences}
In this section, we prove an $\eta$-compactness result which allows one to relate an energy bound to the non-existence of vortices. The idea here is that for two concentric balls, if the energy on the larger ball is small enough, then it is impossible for a vortex to exist in the smaller ball. This fact is pivotal in proving that the set of points $x\in\overline{\Omega}$ for which $|u_{\e}|$ is small can be covered by a finite set of $\e$-balls whose number is bounded independent of $\e$. We begin by stating a Pohosaev-type identity for solutions of \eqref{eq:ELstrong} or \eqref{eq:ELweak} which is obtained via integrating by parts against a smooth function. This identity will be needed to obtain an $\eta$-compactness result which will be developed in the next section.\\\\
Define
\begin{equation*}
e_{\e}(u):=\frac{1}{2}|\nabla u|^2+\frac{1}{4\e^2}\left(1-|u|^2\right)^2.
\end{equation*}
\begin{prop}
Let $\psi\in C^2(\Omega;\mathbb{R}^2)$. If $u$ is a solution of \eqref{eq:ELweak} or \eqref{eq:ELstrong}, then
\begin{equation}\label{Poh}
\int_{\partial\omega_r}\left\{e_{\e}(u)\langle\psi,n\rangle-\langle \partial_nu,\psi\cdot\nabla u\rangle\right\}ds\\=\int_{\omega_r}\left\{e_{\e}(u)\Div \psi-\sum_{j,l}\psi_{x_j}^{l}\langle\partial_{x_j}u,\partial_{x_{l}}u\rangle\right\}dx.
\end{equation}
\end{prop}


\begin{thm}[$\eta$-Compactness]\label{thm:eta}
[Strong Tangental Case] Let $\frac{3}{4}\leq \beta<\gamma<1$. There exists constants $\eta$, $\tilde{C}$, $\e_0>0$ such that for any solution $u_{\e}$ of \eqref{eq:ELstrong} with $\e\in(0,\e_0)$, if $x_0\in\overline{\Omega}$ and $E_{\e}(u_{\e};\omega_{2\e^{\beta}}(x_0))\leq\eta |\ln\e|$, then
\begin{align}
&|u_{\e}|\geq \frac{1}{2}\quad \mbox{in}\ \omega_{\e^{\gamma}}(x_0)\label{ineq:half2},\\[0.5em]
&\frac{1}{4\e^2}\int_{\omega_{\e^{\gamma}}(x_0)}(1-|u_{\e}|^2)^2\,dx\leq \tilde{C}\eta\label{ineq:ceta2}.
\end{align}
[Weak Tangential Case] Let $\frac{3}{4}s\leq\beta<\gamma<s\leq 1$. There exists constants $\eta$, $\tilde{C}$, $\e_0>0$ such that for any solution $u_{\e}$ of \eqref{eq:ELweak} with $\e\in(0,\e_0)$, if $x_0\in\overline{\Omega}$ and $E_{\e}^{g,s}(u_{\e};\omega_{2\e^{\beta}}(x_0))\leq\eta |\ln\e|$, then
\begin{align}
&|u_{\e}|\geq \frac{1}{2}\quad \mbox{in}\ \omega_{\e^{\gamma}}(x_0)\label{ineq:half},\\[0.5em]
&|\langle u_{\e},g^{\perp}\rangle|\leq \frac{1}{4}\quad \mbox{on}\ \Gamma\cap\overline{\omega_{\e^{\gamma}}(x_0)}\label{ineq:quarter},\\[0.5em]
&\frac{1}{4\e^2}\int_{\omega_{\e^{\gamma}}(x_0)}(1-|u_{\e}|^2)^2\,dx+\frac{1}{2\e^s}\int_{\Gamma\cap\overline{\omega_{\e^{\gamma}}(x_0)}}\langle u_{\e},g^{\perp}\rangle^2\,ds\leq \tilde{C}\eta\label{ineq:ceta}.
\end{align}
\end{thm}
The proof for Theorem \ref{thm:eta} is heavily dependent on a crucial estimate. For $x_0\in\overline{\Omega}$, define as in \cite{struwe1994asymptotic, moser2003} the functions
\begin{align*}
F(r)=F(r;x_0,u,\e):=r\int_{\partial B_r(x_0)\cap\Omega}e_{\e}(u)\,ds,\quad F_{\Gamma}(r):=F(r)+\frac{r}{2\e^s}\sum_{x\in\partial\Gamma_r(x_0)}\langle u,g^{\perp}\rangle^2
\end{align*}
where the second function above is defined when $x_0\in\Gamma$.
\begin{lem}\label{lem:estimate}
Let $x_0\in\overline{\Omega}$. There exists constants $C>0$ and $r_0>0$ such that for $\e\in(0,1)$ and $r\in(0,r_0)$ we have:
\begin{enumerate}
\item If $x_0\in\Omega$, $\overline{\omega_r(x_0)}\cap\Gamma=\varnothing$ and $u$ is a solution of either \eqref{eq:ELweak} or \eqref{eq:ELstrong}, then
\begin{equation*}
\frac{1}{4\e^2}\int_{\omega_r}(1-|u|^2)^2\,dx\leq r\int_{\omega_r}\frac{1}{2}|\nabla u|^2 dx+F(r).
\end{equation*}
\item If $x_0\in\Gamma$ and $u$ is a solution of \eqref{eq:ELstrong}, then
\begin{equation}\label{intbound2}
\frac{1}{4\e^2}\int_{\omega_r}(1-|u|^2)^2\,dx\leq C\left[r\int_{\omega_r}\frac{1}{2}|\nabla u|^2\,dx+F(r)+\frac{r^2}{\e}\right].
\end{equation}
\item If $x_0\in\Gamma$ and $u$ is a solution of \eqref{eq:ELweak}, then
\begin{equation}\label{intbound}
\frac{1}{4\e^2}\int_{\omega_r}(1-|u|^2)^2\,dx+\frac{1}{2\e^s}\int_{\Gamma_r}\langle u,g^{\perp}\rangle^2\,ds\leq C\left[r\int_{\omega_r}\frac{1}{2}|\nabla u|^2\,dx+F_{\Gamma}(r)+\frac{r^2}{\e^s}\right].
\end{equation}
\end{enumerate}
\end{lem}
\begin{proof}
The proof for case (1) is shown in \cite[Lemma 2]{struwe1994asymptotic}. Inequality \eqref{intbound} follows from \cite[Lemma 5.4]{moser2003} or by changing every instance of $|u-g|^2$ with $\langle u,g^{\perp}\rangle^2$ throughout \cite[Lemma 4.2]{alama2015weak}. Thus, it only remains to prove \eqref{intbound2}. Let $r_0>0$ be chosen small enough so that $\Gamma\cap B_r(x_0)$ consists of a single smooth arc satisfying $|\Gamma_r|\leq Cr$ for all $0<r\leq r_0$. As in \cite{alama2015weak} we let $\mathcal{N}$ be a $2r_0$-neighbourhood of $\Gamma$, and by taking $r_0$ smaller if necessary, there exists a vector field $X\in C^2(\mathcal{N};\mathbb{R}^2)$ satisfying
\begin{align}
&\langle X,n\rangle=X_n=0\quad \mbox{for all}\ x\in\Gamma_r, \label{Xnormal}\\[0.5em]
&|X-(x-x_0)|\leq C|x-x_0|^2\quad \mbox{for all}\ x\in\omega_r, \label{Xbound}\\[0.5em]
& |\partial_{x_i}X^j-\delta_{ij}|\leq C|x-x_0|\quad \mbox{for all}\ x\in\omega_r, \label{DXbound}
\end{align} 
for a constant $C>0$ and for any $x_0\in\Gamma$. To obtain inequality \eqref{intbound2} we consider the Pohosaev-type identity \eqref{Poh} with $\psi=X$ and  estimate. Using the boundary decomposition $\partial\omega_r=\Gamma_r\cup(\partial B_r(x_0)\cap\Omega)$, it will be convenient to perform these estimates on $\Gamma_r$ and $\partial B_r(x_0)\cap\Omega$ separately.\\\\
\underline{Estimates Along $\Gamma_r$}:\\[0.5em]
By \eqref{Xnormal} we may write $X=\langle X,\tau\rangle\tau=X_{\tau}\tau$ where $\tau$ is the unit tangent vector to $\Gamma_r$ and so $X\cdot\nabla u=X_{\tau}\partial_{\tau}u$ on $\Gamma_r$. With this, the lefthand side of \eqref{Poh} becomes
\begin{equation*}
\int_{\Gamma_r}\left\{e_{\e}(u)X_n-\langle \partial_nu,X\cdot\nabla u\rangle\right\}ds=-\int_{\Gamma_r}\langle \partial_nu ,X_{\tau}\partial_{\tau}u\rangle\,ds.
\end{equation*}
Using the derivative representations 
\begin{align*}
\partial_nu&=\partial_n(u_{\parallel}g+u_{\perp}g^{\perp})=u_{\parallel}\partial_ng+\partial_nu_{\parallel}g+u_{\perp}\partial_ng^{\perp}+\partial_nu_{\perp}g^{\perp},\\[0.5em]
\partial_{\tau}u&=\partial_{\tau}(u_{\parallel}g+u_{\perp}g^{\perp})=u_{\parallel}\partial_{\tau}g+\partial_{\tau}u_{\parallel}g+u_{\perp}\partial_{\tau}g^{\perp}+\partial_{\tau}u_{\perp}g^{\perp},
\end{align*}
with the known conditions $u_{\perp}=\partial_nu_{\parallel}=\partial_{\tau}u_{\perp}=0$, we obtain
\begin{align*}
\langle \partial_nu ,X_{\tau}\partial_{\tau}u\rangle&=X_{\tau}\langle u_{\parallel}\partial_ng+\partial_nu_{\perp}g^{\perp},u_{\parallel}\partial_{\tau}g+\partial_{\tau}u_{\parallel}g\rangle\\
&=X_{\tau}((u_{\parallel})^2\langle\partial_ng,\partial_{\tau}g\rangle+u_{\parallel}\partial_nu_{\perp}\langle g^{\perp},\partial_{\tau}g\rangle).
\end{align*}
Applying Lemma \ref{lem:bounded} and Cauchy-Schwarz,
\begin{align*}
|(u_{\parallel})^2\langle\partial_ng,\partial_{\tau}g\rangle|&\leq |\partial_n g||\partial_{\tau}g|\leq |\nabla g|^2\leq C=C(g),\\[0.5em]
|u_{\parallel}\partial_nu_{\perp}\langle g^{\perp},\partial_{\tau}g\rangle|&\leq |\partial_n u_{\perp}||g^{\perp}||\partial_{\tau}g|\leq |\nabla u||\nabla g|\leq CC_0\e^{-1}.
\end{align*}
Therefore, there is a constant $c$ for which
\begin{equation*}
|\langle \partial_nu ,X_{\tau}\partial_{\tau}u\rangle|\leq |X_{\tau}|\frac{c}{\e}.
\end{equation*}
Moreover since $|X_{\tau}|\leq Cr$ and $|\Gamma_r|\leq Cr$ we have another constant $C$ (independent of $\e$) so that
\begin{equation*}
\left|\int_{\Gamma_r}\langle\partial_nu,X\cdot\nabla u\rangle\,ds\right|\leq \int_{\Gamma_r}|X_{\tau}|\frac{c}{\e}\,ds\leq \frac{Cr^2}{\e}.
\end{equation*}
\underline{Estimates Along $\partial B_r(x_0)\cap\Omega$}:\\[0.5em]
The lefthand side of (\ref{Poh}) along $\partial B_r(x_0)\cap\Omega$ can be written as the sum of integrals $I_1+I_2$ which we estimate separately. First, by \eqref{Xbound}, $|X_n|$, $|X_{\tau}|\leq Cr$ and by applying Cauchy-Schwarz we get
\begin{align*}
I_1&=\int_{\partial B_r(x_0)\cap\Omega}\left\{\frac{1}{2}|\nabla u|^2X_n-X_n|\partial_nu|^2-X_{\tau}\langle\partial_nu,\partial_{\tau}u\rangle\right\}ds\\[0.5em]
&\leq  Cr\int_{\partial B_r(x_0)\cap\Omega}\left\{\frac{1}{2}|\partial_{\tau}u|^2+\frac{1}{2}|\partial_nu|^2+\frac{1}{2}|\partial_nu|^2+\frac{1}{2}|\partial_{\tau}u|^2\right\}ds\\[0.5em]
&=Cr\int_{\partial B_r(x_0)\cap\Omega}|\nabla u|^2\,ds.
\end{align*}
An easy estimate for $I_2$ is given by
\begin{equation*}
I_2=\frac{1}{4\e^2}\int_{\partial B_r(x_0)\cap\Omega}(1-|u|^2)^2X_n\,ds\leq \frac{Cr}{4\e^2}\int_{\partial B_r(x_0)\cap\Omega}(1-|u|^2)^2\,ds.
\end{equation*}
Thus, for $C>0$ large enough
\begin{align*}
I_1+I_2\leq Cr\int_{\partial B_r(x_0)\cap\Omega}\frac{1}{2}\left\{|\nabla u|^2+\frac{1}{2\e^2}(1-|u|^2)^2\right\}ds=CF(r)
\end{align*}
and so the lefthand side of \eqref{Poh} has the estimate
\begin{align*}
\int_{\partial\omega_r}\left\{e_{\e}(u)X_n-\langle\partial_nu,X\cdot\nabla u\rangle\right\}ds&=I_1+I_2-\int_{\Gamma_r}\langle\partial_nu,X\cdot\nabla u\rangle\,ds\\[0.5em]
&\leq C\left[F(r)+\frac{r^2}{\e}\right].
\end{align*}
\underline{Estimates on $\omega_r$}:\\[0.5em]
The righthand side of \eqref{Poh} can be written as the sum of integrals $J_1+J_2$ which again we estimate separately. By \eqref{DXbound} and Cauchy-Schwarz,
\begin{equation*}
\sum_{j,l}X_{x_j}^{l}\langle\partial_{x_j}u,\partial_{x_{l}}u\rangle\leq|\nabla u|^2+2 Cr|\nabla u|^2
\end{equation*}
and since $\Div X\geq 2 - 2Cr$ we have 
\begin{align*}
J_1&=\int_{\omega_r}\left\{\frac{1}{2}|\nabla u|^2\Div X-\sum_{j,l}X_{x_j}^{l}\langle\partial_{x_j}u,\partial_{x_{l}}u\rangle\right\}dx\\
&\geq \int_{\omega_r}\left\{\frac{1}{2}|\nabla u|^2\Div X-|\nabla u|^2-2 Cr|\nabla u|^2\right\}dx\\
&\geq-Cr\int_{\omega_r}|\nabla u|^2\,dx.
\end{align*} 
For the integral $J_2,$ one can choose $r_0$ smaller if necessary so that $\Div X\geq 2-2Cr\geq 1$ and
\begin{align*}
J_2=\frac{1}{4\e^2}\int_{\omega_r}(1-|u|^2)^2\Div X\,dx\geq \frac{1}{4\e^2}\int_{\omega_r}(1-|u|^2)^2\,dx.
\end{align*}
Putting these estimates together,
\begin{equation*}
\frac{1}{4\e^2}\int_{\omega_r}(1-|u|^2)^2\,dx-Cr\int_{\omega_r}\frac{1}{2}|\nabla u|^2\,dx\leq J_1+J_2=I_1+I_2\leq C\left[F(r)+\frac{r^2}{\e}\right]
\end{equation*}
which completes the proof for inequality \eqref{intbound2}.
\end{proof}
We now prove Theorem \ref{thm:eta}.
\begin{proof}[Proof of Theorem \ref{thm:eta}.]
The case where $x_0\in\Omega$ and $\omega_{2\e^{\beta}}(x_0)\cap\Gamma=\varnothing$ is shown in \cite[Lemma 2.3]{struwe1994asymptotic}. Therefore, it is sufficient to prove the result for $x_0\in\Gamma$. We begin by proving \eqref{ineq:ceta2} for strong tangential solutions. Inequality \eqref{ineq:ceta} for weak tangential solutions is done similarly. Using the mean value theorem for integrals, there exists $r_{\e}\in(2\e^{\gamma},2\e^{\beta})$ such that $F(r_{\e})\leq \eta(\gamma-\beta)^{-1}$. Using the radius $r=r_{\e}$ in \eqref{intbound2} gives
\begin{align*}
\frac{1}{4\e^2}\int_{\omega_{r_{\e}}(x_0)}(1-|u_{\e}|^2)^2\,dx&\leq C\left[r_{\e}\int_{\omega_{r_{\e}}(x_0)}\frac{1}{2}|\nabla u|^2\,dx+F(r_{\e})+\frac{r_{\e}^2}{\e}\right]\\[0.5em]
&\leq C\left[2\e^{3/4}\eta|\ln\e|+\frac{\eta}{\gamma-\beta}+4\sqrt{\e}\right].
\end{align*}
Then for $\e<\e_0$ with appropriately chosen $\e_0>0$, we get \eqref{ineq:ceta2}. Inequalities \eqref{ineq:half2} and \eqref{ineq:half} can be obtained using a contradiction argument. By assuming there is some $x_1\in \omega_{\e^{\gamma}}(x_0)$ such that $|u(x_1)|<1/2$, standard methods involving the mean value theorem and smoothness properties of $\Gamma$ \cite{bbh} allow us to conclude that there is a radius $r=c\e$ and a constant $c'>0$ independent of $\eta$ and $\e$ such that
\begin{equation*}
\tilde{C}\eta\geq \frac{1}{4\e^2}\int_{\omega_{2\e^{\gamma}}(x_0)}(1-|u|^2)^2 dx\geq \frac{1}{4\e^2}\int_{\omega_{c\e}(x_1)}(1-|u|^2)^2 dx\geq c'.
\end{equation*}
Taking $\eta$ smaller if necessary gives the contradiction.
For inequality \eqref{ineq:quarter}, suppose $r_0>0$ is taken small enough so that $\omega_r(x_0)$ is strictly starshaped with respect to some point $x_2\in \omega_r(x_0)$. By setting $\psi=x-x_2$ in \eqref{Poh} and following \cite[Proposition 4.1]{alama2015weak} or \cite[Proposition 4.1]{alama2020thin}, one can find the estimate
\begin{equation}\label{ineq:difference}
|u_{\e}(x)-u_{\e}(y)|\leq C\sqrt{|x-y|}\e^{-s/2}
\end{equation}
holding for all $x,y\in\Gamma_{r_{\e}}$ with $C$ a constant independent of $\e$ and $x_0$ . As in the interior case, suppose there is some point $x_3\in\Gamma_{r_{\e}}$ such that $|\langle u_{\e}(x_3),g^{\perp}(x_3)\rangle|>1/4$. Then by the triangle inequality, Cauchy-Schwarz and the uniform bound of Lemma \ref{lem:bounded}:
\begin{equation*}
|\langle u_{\e}(x),g^{\perp}(x)\rangle|>1/4-|u_{\e}(x)-u_{\e}(x_3)|-|g^{\perp}(x)-g^{\perp}(x_3)|.
\end{equation*}
By \eqref{ineq:difference} and the smoothness of $g$, a radius $\rho$ proportional to $\e^s$ can be chosen so that
\begin{equation*}
|u_{\e}(x)-u_{\e}(x_3)|,\ |g^{\perp}(x)-g^{\perp}(x_3)|<\frac{1}{16}
\end{equation*}
for all $x\in\Gamma_{r_{\e}}\cap B_{\rho}(x_3)$ and so $|\langle u_{\e}(x),g^{\perp}(x)\rangle|>1/8$ on this set.
Applying inequality \eqref{ineq:ceta} along with the estimate $|\Gamma_{r_{\e}}\cap B_{\rho}(x_3)|\geq c''\e^s$ with $c''$ independent of $\e$ and $\eta$, we have
\begin{equation*}
\tilde{C}\eta\geq \frac{1}{2\e^s}\int_{\Gamma_{r_{\e}}\cap B_{\rho}(x_3)}\langle u_{\e},g^{\perp}\rangle^2\,ds>\frac{c''}{128}.
\end{equation*}
As before, we choose $\eta$ small enough to obtain a contradiction.
\end{proof}
Define the family of bad sets
\begin{align*}
S_{\e}&:=\left\{x\in\overline{\Omega}:|u_{\e}(x)|<\frac{1}{2}\right\},\\
S_{\e}^{g,s}&:=S_{\e}\cup\left\{x\in\Gamma:|\langle u_{\e}(x),g^{\perp}(x)\rangle|>\frac{1}{4}\right\}.
\end{align*}
\begin{prop}\label{prop:uniformballbound}
[Strong Tangential Case] There exists $\tilde{N}\in\mathbb{N}$ depending only on $\Omega$, a constant $\lambda>1$ independent of $\e$ and points $p_{\e,1},\ldots,p_{\e,I_{\e}}\in S_{\e}\cap\Omega$, $q_{\e,1},\ldots,q_{\e,J_{\e}}\in S_{\e}\cap\Gamma$ such that
\begin{enumerate}[(i)]
\item $I_{\e}+J_{\e}\leq \tilde{N}$,
\item $S_{\e}\subset \bigcup_{i=1}^{I_{\e}}B_{\lambda\e}(p_{\e,i})\cup \bigcup_{j=1}^{J_{\e}}B_{\lambda\e}(q_{\e,j})$,
\item $\{B_{\lambda\e}(p_{\e,i}),B_{\lambda\e}(q_{\e,j})\}_{1\leq i\leq I_{\e},1\leq j\leq J_{\e}}$ are mutually disjoint with centers satisfying
\begin{equation*}
|p_{\e,i}-p_{\e,j}|,\ |q_{\e,i}-q_{\e,j}|,\ |p_{\e,i}-q_{\e,j}|> 8\lambda \e,
\end{equation*}
\item $\overline{B_{\lambda\e}(p_{\e,i})}\cap\Gamma=\varnothing$ for all $i=1,\ldots, I_{\e}$.
\end{enumerate}
[Weak Tangential Case] There exists $\tilde{N}\in\mathbb{N}$ depending only on $\Omega$, a constant $\lambda>1$ independent of $\e$ and points $p_{\e,1},\ldots,p_{\e,I_{\e}}\in S_{\e}^{g,s}\cap\Omega$, $q_{\e,1},\ldots,q_{\e,J_{\e}}\in S_{\e}^{g,s}\cap\Gamma$ such that
\begin{enumerate}[(i)]
\item $I_{\e}+J_{\e}\leq \tilde{N}$,
\item $S_{\e}^{g,s}\subset \bigcup_{i=1}^{I_{\e}}B_{\lambda\e}(p_{\e,i})\cup \bigcup_{j=1}^{J_{\e}}B_{\lambda\e^s}(q_{\e,j})$,
\item $\{B_{\lambda\e}(p_{\e,i}),B_{\lambda\e^s}(q_{\e,j})\}_{1\leq i\leq I_{\e},1\leq j\leq J_{\e}}$ are mutually disjoint with centers satisfying
\begin{equation*}
|p_{\e,i}-p_{\e,j}|> 8\lambda \e,\quad |q_{\e,i}-q_{\e,j}|>8\lambda \e^s,\quad and \quad |p_{\e,i}-q_{\e,j}|> 8\lambda \e^s,
\end{equation*}
\item $\overline{B_{\lambda\e}(p_{\e,i})}\cap\Gamma=\varnothing$ for all $i=1,\ldots, I_{\e}$.
\end{enumerate}
\end{prop}
The proof is exactly as in \cite[Lemma 4.4]{alama2020thin} which is based on the method of \cite[Section 3]{struwe1994asymptotic} and a ball merging method which is presented in \cite[Theorem IV.1] {bbh}. As a consequence of Balzano-Weierstrass, we also have a result which states that the bad sets $S_{\e}$, $S_{\e}^{g,s}$ can eventually be covered by a static ball covering (along a subsequence $\e_n\to 0$).
\begin{prop}\label{prop:sigmaballs}
\text{[Strong Tangential Case]} For any sequence of $\e\to 0$ there is a subsequence $\e_n\to 0$, a constant $\sigma_0>0$ and a finite collection of points $\{p_1,\ldots,p_I\}\subset\Omega$, $\{q_1,\ldots,q_J\}\subset\Gamma$ such that for any $0<\sigma<\sigma_0$ and for all $n\in\mathbb{N}$, the collection of sets
\begin{equation}\label{def:sigma}
\mathcal{S}_{\sigma}:=\{B_{\sigma}(p_i)\}_{i=1}^I\cup\{B_{\sigma}(q_j)\}_{j=1}^J
\end{equation} 
are mutually disjoint and cover $S_{\e_n}$.\\[0.5em]
[Weak Tangential Case] The same result holds for the bad set $S_{\e_n}^{g,s}$ but with $\mathcal{S}_{\sigma}$ replaced by
\begin{equation}\label{def:ssigma}
\mathcal{S}_{\sigma}^{g,s}:=\{B_{\sigma}(p_i)\}_{i=1}^I\cup\{B_{\sigma^s}(q_j)\}_{j=1}^J.
\end{equation} 
\end{prop}


\section{Local Orientation and Defect Windings}\label{sec:wind}
Now that the bad sets $S_{\e}$ and $S_{\e}^{g,s}$ have been shown to have finite bad ball coverings, we are in a position to analyze the winding behaviour of minimizers around defects. When dealing with interior bad balls, we may quantify the winding of $u_{\e}$ on $\partial B_{\lambda\e}(p_{i,\e})$ in the usual way since $|u_{\e}|\geq 1/2$ on this curve. In particular, we define the degree of $u_{\e}$ around $\partial B_{\lambda\e}(p_{i,\e})$ to be the degree of the normalization of $u_{\e}$ about $\partial B_{\lambda\e}(p_{i,\e})$:
\begin{equation*}
d_i=d_{i,\e}=\deg(u_{\e};\partial B_{\lambda\e}(p_{i,\e})):=\deg(u_{\e}/|u_{\e}|;\partial B_{\lambda\e}(p_{i,\e})).
\end{equation*}
Analyzing the winding of $u_{\e}$ about boundary bad balls is slightly more subtle, however. As mentioned in the introduction, we define the notion of a \emph{boundary index}, whose function is analogous to the degree of interior defects. Specifically, the boundary index aims to quantify the turning behaviour of $u_{\e}$ along circular arcs lying in the interior $\Omega$ that connect two nearby points on $\Gamma$. To begin constructing this quantity, we place focus on strong tangential solutions and then show the necessary modifications for weak tangential solutions.\\

Consider again the local polar coordinate system found in Section \ref{sec:prelim} as defined by the angular bounds \eqref{bounds:theta} with center point $q_{\e,j}\in\Gamma$. Let $B_{\lambda\e}(q_{\e,j})$ be some fixed boundary bad ball for a strong tangential solution $u_{\e}$ and fix $R>\lambda\e$ so that for all $\rho\in[\lambda\e, R]$ the closure of $\omega_{\rho}(q_{\e,j})$ does not intersect the closure of any other bad ball. Since $|u_{\e}|\geq 1/2$ outside $B_{\lambda\e}(q_{\e,j})$, there is a single-valued function $\psi$ with 
\begin{equation*}
\frac{u_{\e}}{|u_{\e}|}=e^{i\psi}\quad \mbox{on}\ \overline{A_{\lambda\e,R}(q_{\e,j})}.
\end{equation*}
Likewise, there is a lifting $\gamma$ of $g$ for which $g=e^{i\gamma}$ on $\overline{A_{\lambda\e,R}(q_{\e,j})}$ since $g:\mathcal{N}_{\Gamma}\to\mathbb{S}^1$. Looking along the curve $\Gamma^{+}_{\lambda\e,R}(q_{\e,j})$ and using the definition $\mathcal{H}_g(\Omega)$, it holds that either
\begin{equation}\label{bd:gplusstrong}
|\psi(\rho,\theta_1(\rho))-\gamma(\rho,\theta_1(\rho))|=0\quad \mbox{or}\quad |\psi(\rho,\theta_1(\rho))-\gamma(\rho,\theta_1(\rho))\pm\pi|=0.
\end{equation}
These two possible conditions comes down the observation that $u_{\e}$ has either a phase equal to that of $g$ or to $-g$. This fact induces a sense of orientation for $u_{\e}$ with respect to $g$ along boundary components outside of bad balls. In particular, we say that $u_{\e}$ is positively oriented with respect to $g$ at $x\in\Gamma$ (p.o.) provided $\langle u(x),g(x)\rangle>0$ (which corresponds to the first condition of \eqref{bd:gplusstrong}) and that $u_{\e}$ is negatively oriented (n.o.) in the opposite case $\langle u(x),g(x)\rangle<0$ (the second condition of \eqref{bd:gplusstrong}).\\
\begin{figure}[h]
    \centering
    \includegraphics{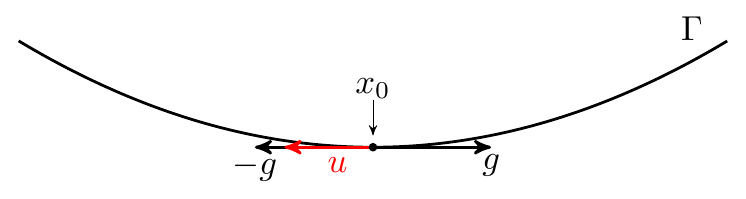} 
    \caption{Depiction of negative orientation with respect to $g$ for a strong tangential solution $u$ at a point $x_0\in\Gamma$ away from bad balls.} 
\end{figure}

Upon starting at the point
\begin{equation*}
q_{\e,j}^+=q_{\e,j}^+(\rho):=\partial B_{\rho}(q_{\e,j})\cap \Gamma_{\lambda\e,R}^+(q_{\e,j})
\end{equation*}
and traveling along the arc $\partial B_{\rho}(q_{\e,j})\cap\Omega$, one arrives at the point
\begin{equation*}
q_{\e,j}^-=q_{\e,j}^-(\rho):=\partial B_{\rho}(q_{\e,j})\cap \Gamma_{\lambda\e,R}^-(q_{\e,j})
\end{equation*}
for which $u_{\e}$ has accumulated an approximate net number of $\pi$-rotations. More precisely, there exists a unique integer $k\in\mathbb{Z}$ such that
\begin{equation}\label{bd:gnegstrong}
|\psi(\rho,\theta_2(\rho))-\gamma(\rho,\theta_2(\rho))-k\pi|=0\quad \mbox{on}\ \Gamma_{\lambda\e,R}^{-}(q_{\e,j}).
\end{equation}
With this information, we may now define the boundary index. Let $w:A_{\lambda\e,R}(q_{\e,j})\to\mathbb{S}^1$ be defined by
\begin{equation*}
w:=\left(\frac{u_{\e}}{|u_{\e}|}\right)^2=e^{2i\psi}.
\end{equation*}
Using \eqref{bd:gplusstrong} and \eqref{bd:gnegstrong},
\begin{align*}
\left\{ 
\begin{array}{ll}
|\arg(w)-2\gamma|=2|\psi-\gamma|=0 & \mbox{on $\Gamma_{\lambda\e,R}^{+}(q_{\e,j})$ if $u_{\e}$ is p.o.,}\\[0.5em]
|\arg(w)-2\gamma\pm2\pi|=2|\psi-\gamma\pm\pi|=0 & \mbox{on $\Gamma_{\lambda\e,R}^{+}(q_{\e,j})$ if $u_{\e}$ is n.o.,}\\[0.5em]
|\arg(w)-2\gamma-2\pi k|=2|\psi-\gamma-k\pi|=0 & \mbox{on $\Gamma_{\lambda\e,R}^{-}(q_{\e,j})$},
\end{array}
\right.
\end{align*}
and so $w$ has preserved orientation on $\Gamma^{\pm}_{\lambda\e,R}(q_{\e,j})$ with respect to $g^2$. Therefore we may extend $w$ to $x\in \Gamma_{\lambda\e}(q_{\e,j})$ as an $\mathbb{S}^1$-valued, piecewise $C^2$ map by simply setting $w=g^2$ along $\Gamma_{\lambda\e}(q_{\e,j})$. The reader can refer to Figure \ref{fig:w} for an illustration.
\begin{figure}[h]
  \begin{subfigure}[b]{0.5\linewidth}
    \centering
    \includegraphics[width=0.95\linewidth]{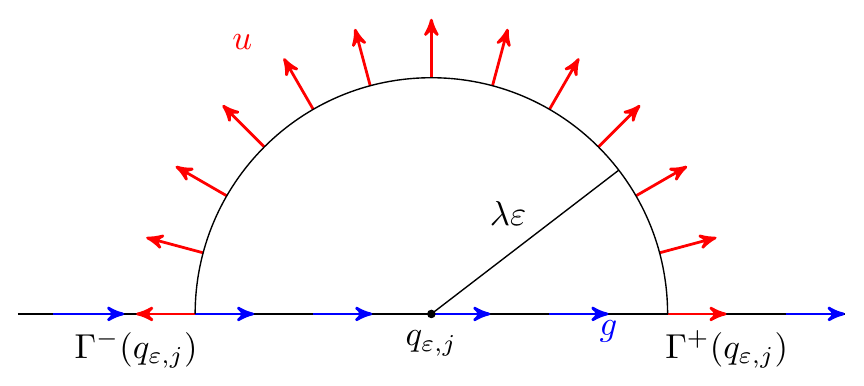} 
    \caption{Possible profile of a strong tangential\\ solution $u$ outside a boundary bad ball} 
    \vspace{1ex}
  \end{subfigure}
  \begin{subfigure}[b]{0.5\linewidth}
    \centering
    \includegraphics[width=0.95\linewidth]{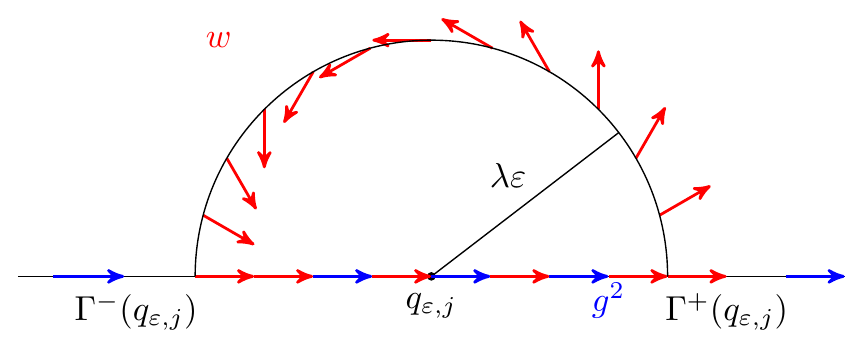} 
    \caption{Corresponding profile of $w$ outside a \\ boundary bad ball with interior extension} 
    \vspace{1ex}
  \end{subfigure} 
 \caption[Orientations]{Relationship between $u$, $w$, $g$ and $g^2$.}
 \label{fig:w}
\end{figure}
Next, for any $\rho\in[\lambda\e,R]$, form the closed curve
\begin{equation*}
C_{\rho}:=(\partial B_{\rho}(q_{\e,j})\cap\Omega)\cup \Gamma_{\rho}(q_{\e,j})
\end{equation*}
with positive orientation. By construction of $w$ we may define
\begin{equation}\label{def:bdryindex}
D_j=D_j(q_{\e,j}):=\deg(w;C_{\rho})
\end{equation}
whose value is independent of $\rho\in[\lambda,R]$ for any particular extension of $w$ by properties of the degree. Returning to the local polar coordinate system $(\rho,\theta)$ centered at $q_{\e,j}$, we note that
\begin{equation*}
\arg(w)=2\psi=2D_j\theta+\tilde{\phi}(\rho,\theta)
\end{equation*}
where $\tilde{\phi}$ is a single-valued function in $A_{\lambda\e,R}(q_{\e,j})$. Therefore $u_{\e}=f(\rho,\theta)e^{i\psi}$ with
\begin{equation}\label{brphase}
\psi(\rho,\theta)=D_j\theta+\phi(\rho,\theta)
\end{equation}
and $\phi$ a single-valued function in $\overline{A_{\lambda\e,R}(q_{\e,j})}$. The integer $D_j\in\mathbb{Z}$ is what we define as the boundary index. In some cases, we will use the notation 
\begin{equation*}
D_j:=\ind(u_{\e};\partial B_{\rho}(q_{\e,j})\cap\Omega).
\end{equation*}
The boundary index for weakly tangential solutions can be constructed in the same way, but now with boundary bad balls having radii $\rho=\lambda\e^s$. However, the phase of $u_{\e}$ along $\Gamma_{\lambda\e^s,R}^{\pm}(q_{\e,j})$ no longer satisfies the strict conditions of \eqref{bd:gplusstrong} and \eqref{bd:gnegstrong}. In this case, the definition of $S_{\e}^{g,s}$ can be used to show that
\begin{equation}\label{bd:gplus}
|\psi(\rho,\theta_1(\rho))-\gamma(\rho,\theta_1(\rho))|<\frac{\pi}{6}\quad \mbox{or}\quad |\psi(\rho,\theta_1(\rho))-\gamma(\rho,\theta_1(\rho))\pm\pi|<\frac{\pi}{6}
\end{equation}
on $\Gamma_{\lambda\e^s,R}^{+}(q_{\e,j})$ (depending on the orientation of $u$ with respect to $g$), and that there is a unique integer $k\in\mathbb{Z}$ such that
\begin{equation}\label{bd:gneg}
|\psi(\rho,\theta_2(\rho))-\gamma(\rho,\theta_2(\rho))-k\pi|<\frac{\pi}{6}\quad \mbox{on}\ \Gamma_{\tilde{r},R}^{-}(q_{\e,j}).
\end{equation}
\begin{figure}[h]
\includegraphics{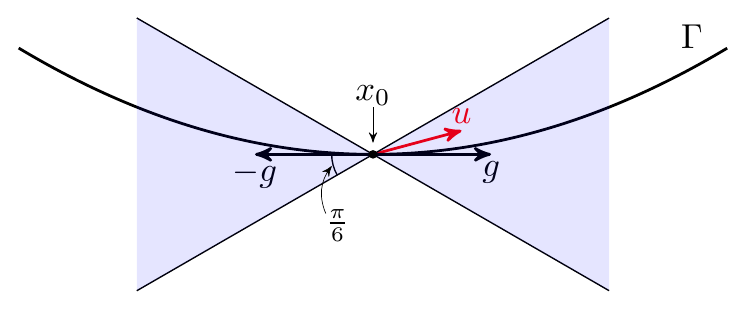}
\caption{Illustration of positive orientation with respect to $g$ for a weakly tangential solution $u$ at a point $x_0\in\Gamma$ away from bad balls. By definition of $S_{\e}^{g,s}$, the vector $u(x_0)$ may only reside within the shaded double cone with axis defined by $g$.}
\end{figure}\\
Again, by defining the $\mathbb{S}^1$-valued function $w=(u/|u|)^2$ on $A_{\lambda\e^s,R}(q_{\e,j})$, from \eqref{bd:gplus} and \eqref{bd:gneg} we have
\begin{align*}
\left\{ 
\begin{array}{ll}
|\arg(w)-2\gamma|=2|\psi-\gamma|<\frac{\pi}{3} & \mbox{on $\Gamma_{\lambda\e^s,R}^{+}(q_{\e,j})$ if $u_{\e}$ is p.o.,}\\[0.5em]
|\arg(w)-2\gamma\pm2\pi|=2|\psi-\gamma\pm\pi|<\frac{\pi}{3} & \mbox{on $\Gamma_{\lambda\e^s,R}^{+}(q_{\e,j})$ if $u_{\e}$ is n.o.,}\\[0.5em]
|\arg(w)-2\gamma-2\pi k|=2|\psi-\gamma-k\pi|<\frac{\pi}{3} & \mbox{on $\Gamma_{\lambda\e^s,R}^{-}(q_{\e,j})$},
\end{array}
\right.
\end{align*}
which as before, shows that the orientation of $w$ with respect to $g^2$ is preserved on $\Gamma_{\lambda\e^s,R}^{\pm}(q_{\e,j})$. We can now extend $w$ to $x\in \Gamma_{\lambda\e^s}(q_{\e,j})$ as an $\mathbb{S}^1$-valued, piecewise $C^2$ map satisfying $|\arg(w)-2\gamma|<\pi/3$ which can be done via interpolating the phase linearly across $\Gamma_{\lambda\e^s}(q_{\e,j})$, for example. The boundary index can now be defined in the same way as the strong tangential case.\\

The first main identity we obtain from this definition connects the sum of all associated bad ball boundary indices with $\mathcal{D}=\deg(g;\Gamma)$ and the sum of degrees for the interior bad balls.
\begin{prop}\label{prop:winding}
[Strong Tangential Case] Suppose $u_{\e}$ is a solution of \eqref{eq:ELstrong} with associated bad ball covering $\{B_{\lambda\e}(p_{\e,i}),B_{\lambda\e}(q_{\e,j})\}_{1\leq i\leq I_{\e},1\leq j\leq J_{\e}}$. Let
\begin{align*}
d_i=\deg(u_{\e};\partial B_{\lambda\e}(p_{\e,i}))\quad\mbox{and}\quad D_j=\ind(u_{\e};\partial B_{\lambda\e}(q_{\e,j})\cap\Omega)
\end{align*}
be the degrees and boundary indices for $u_{\e}$ about its interior and boundary bad balls respectively. Then
\begin{equation}\label{eq:windingidentity}
\mathcal{D}=\sum_{i=1}^{I_{\e}}d_i+\frac{1}{2}\sum_{j=1}^{J_{\e}}D_j.
\end{equation}
[Weak Tangential Case] Identity \eqref{eq:windingidentity} holds for solutions of \eqref{eq:ELweak} and its associated bad ball covering $\{B_{\lambda\e}(p_{\e,i}),B_{\lambda\e^s}(q_{\e,j})\}_{1\leq i\leq I_{\e},1\leq j\leq J_{\e}}$.
\end{prop}
The proof for the weak tangential case is done identically to that of the strong tangential case simply by replacing the radii $\lambda\e$ of the boundary bad balls with $\lambda\e^s$. Thus, we provide a proof only for strong tangential solutions. 
\begin{proof}
Define the domain
\begin{equation*}\label{def:punctureddomain}
\tilde{\Omega}:=\Omega\setminus\left\{\bigcup_{j=1}^{J_{\e}}\overline{\omega_{\lambda\e}(q_{\e,j})}\right\}
\end{equation*}
and let $\tilde{\Gamma}=\partial\tilde{\Omega}$, $C(q_{\e,j})=\partial\omega_{\lambda\e}(q_{\e,j})\cap\Omega$. As in the definition of boundary index \eqref{def:bdryindex}, the function $w=(u_{\e}/|u_{\e}|)^2$ is defined on $\tilde{\Omega}$ and $\tilde{\Gamma}$ and can be extended across each segment $\Gamma_{\lambda\e}(q_{\e,j})$. By the construction of the extension $w$, we have $\deg(w;\Gamma)=\deg(g^2;\Gamma)=2\mathcal{D}$ and so by the definition of boundary index,
\begin{align*}
\deg(w;\tilde{\Gamma})&=\frac{1}{2\pi}\int_{\Gamma\setminus\cup_{j}\Gamma_{\lambda\e}(q_j)}(iw,\partial_{\tau}w)\,ds+\sum_{j=1}^{J_{\e}}\frac{1}{2\pi}\int_{C(q_j)}(iw,\partial_{\tau}w)\,ds\\
&=\frac{1}{2\pi}\int_{\Gamma}(iw,\partial_{\tau}w)\,ds-\sum_{j=1}^{J_{\e}}\frac{1}{2\pi}\int_{\partial\omega_{\lambda\e}(q_j)}(iw,\partial_{\tau}w)\,ds\\
&=\deg(w;\Gamma)-\sum_{j=1}^{J_{\e}}D_j=2\mathcal{D}-\sum_{j=1}^{J_{\e}}D_j.
\end{align*}
Lastly, the vortices $p_{\e,i}$ are contained inside $\tilde{\Gamma}$ and so $\deg(w;\tilde{\Gamma})=\sum_{i=1}^{I_{\e}}2d_i$ where we note that the degree along each interior bad ball is doubled by the definition of $w$. Thus, we obtain \eqref{eq:windingidentity} by equating the two quantities for $\deg(w;\tilde{\Gamma})$ and dividing by $2$. 
\end{proof}
We also have a local summation property holding between degrees and boundary indices:
\begin{lem}\label{lem:addthemup}
[Strong Tangential Case] Let $\mathcal{I}$ and $\mathcal{J}$ be sets of indices for a collection of bad balls $\{B_{\lambda \e}(p_{\e,i})\}_{i\in\mathcal{I}}\cup\{B_{\lambda\e}(q_{\e,j})\}_{j\in\mathcal{J}}$ for a strongly tangential solution $u$ and suppose there is a point $y_0\in\Gamma$ and radius $R>0$ such that the ball $\mathscr{B}_R(y_0)$ satisfies
\begin{equation*}
\left(\bigcup_{i\in \mathcal{I}}B_{\lambda\e}(p_{\e,i})\cup \bigcup_{j\in \mathcal{J}}B_{\lambda\e}(q_{\e,j})\right)\subset \mathscr{B}_R(y_0)
\end{equation*}
where $\overline{\mathscr{B}_R(y_0)}$ does not intersect the closure of any other bad ball. Then if $\mathscr{D}=\ind(u_{\e};\partial\mathscr{B}_R(y_0)\cap\Omega)$, $d_i=\deg(u_{\e};\partial B_{\lambda\e}(p_{\e,i}))$ and $D_j=\ind(u_{\e};\partial B_{\lambda\e}(q_{\e,j})\cap\Omega)$, we have
\begin{equation*}
\mathscr{D}=\sum_{j\in \mathcal{J}}D_j+2\sum_{i\in\mathcal{I}}d_i.
\end{equation*}
[Weak Tangential Case] Under the same hypotheses but with boundary bad ball radii replaced by $\lambda\e^s$, the above identity also holds for a collection of bad balls $\{B_{\lambda \e}(p_{\e,i})\}_{i\in\mathcal{I}}\cup\{B_{\lambda\e^s}(q_{\e,j})\}_{j\in\mathcal{J}}$ for weakly tangential solutions. 
\end{lem}
The proof of this Lemma follows the same lines as Proposition \ref{prop:winding}, but can also be shown using longer methods found in \cite[Lemma 4.7]{vanbrussel22}. As above, the proof for the weak tangential case is done identically to that of the strong tangential case by replacing the radii scaling of the boundary bad balls accordingly. To this end, we proceed with the strong tangential case only. 
\begin{proof}
Let $\tilde{\Omega}=\mathscr{B}_R(y_0)\setminus \cup_{j\in\mathcal{J}}\overline{\omega_{\lambda\e}(q_j)}$, $\tilde{\Gamma}=\partial\tilde{\Omega}$ and $C(q_j)=\partial\omega_{\lambda\e}(q_j)\cap\Omega$. Since $|u|\geq 1/2$ in $\tilde{\Omega}$, the function $w=(u/|u|)^2$ is defined on $\overline{\tilde{\Omega}}$. As in the construction of the boundary index, $w$ can be appropriately extended across each segment $\Gamma_{\lambda\e}(q_j)$, $j\in\mathcal{J}$, and so we take $w$ to be defined on $\tilde{\Gamma}\cup_{j\in\mathcal{J}}\Gamma_{\lambda\e}(q_j)$. In particular, this extension of $w$ is defined on $\partial\mathscr{B}_R(y_0)$ and so by definition of the boundary index,
\begin{equation}\label{def:scriptD}
\mathscr{D}=\ind(u;\partial\mathscr{B}_R(y_0)\cap\Omega)=\deg(w;\partial \mathscr{B}_{R}(y_0)).
\end{equation}
Calculating the degree of $w$ along $\tilde{\Gamma}$, we have by \eqref{def:scriptD}
\begin{align*}
\deg(w;\tilde{\Gamma})&=\frac{1}{2\pi}\int_{\partial\mathscr{B}_{R}(y_0)\setminus\cup_{j}\Gamma_{\lambda\e}(q_j)}(iw,\partial_{\tau}w)\,ds+\sum_{j\in\mathcal{J}}\frac{1}{2\pi}\int_{C(q_j)}(iw,\partial_{\tau}w)\,ds\\
&=\frac{1}{2\pi}\int_{\partial\mathscr{B}_{R}(y_0)}(iw,\partial_{\tau}w)\,ds-\sum_{j\in\mathcal{J}}\frac{1}{2\pi}\int_{\partial\omega_{\lambda\e}(q_j)}(iw,\partial_{\tau}w)\,ds\\
&=\mathscr{D}-\sum_{j\in\mathcal{J}}D_j.
\end{align*}
On the other hand, each circle $\partial\omega_{\lambda\e}(p_i)$, $i\in\mathcal{I}$, is contained inside $\tilde{\Gamma}$ and so
\begin{equation*}
\deg(w;\tilde{\Gamma})=\sum_{i\in\mathcal{I}}\deg(w;\partial\omega_{\lambda\e}(p_i))=2\sum_{i\in\mathcal{I}}d_i
\end{equation*}
proving the desired identity. 
\end{proof}
\section{Lower Bounds for the Energy and Convergence}
In this section, we display the details needed to complete the proof of Theorems \ref{thm:big1} and \ref{thm:big2}, which mainly comes down to providing a lower bound for the energies $E_{\e}$ and $E_{\e}^{g,s}$ on the ball collections $\mathcal{S}_{\sigma}$ and $\mathcal{S}_{\sigma}^{g,s}$ respectively. This result is given in Lemma \ref{lem:onsigmas} at the end of this section. The first step in our analysis is to calculate the cost of a vortex locally on annular regions $A_{r,R}$ with $r<R$. To do this, it is useful to begin by characterizing solutions of \eqref{eq:ELstrong} and \eqref{eq:ELweak} in terms of a local polar representation with central point $x_0\in\overline{\Omega}$. Depending on whether $x_0\in\Omega$ or $x_0\in\Gamma$, the representation for the phase of $u_{\e}$ will look slightly different. In any case, the general form for $u_{\e}$ on $A_{r,R}(x_0)$,\ $x_0\in\overline{\Omega}$ can be given by
\begin{equation*}
u_{\e}(\rho,\theta)=f(\rho,\theta)e^{i\psi(\rho,\theta)}\ \mbox{on}\ A_{r,R}(x_0),\ \rho\in [r,R]
\end{equation*}
where $\rho=|x-x_0|$, $\theta$ is an appropriately chosen polar angle and $f(\rho,\theta)=|u_{\e}|$. In the specific case when $x_0\in\Gamma$ there are four general scenarios which can occur for solutions on $\Gamma^{\pm}_{r,R}(x_0)$ when $\langle u,g\rangle \neq 0$:
\begin{enumerate}[(a)]
\item $u$ is p.o.\@ on $\Gamma_{r,R}^+$ and n.o.\@ on $\Gamma_{r,R}^-$,
\item $u$ is p.o.\@ on $\Gamma_{r,R}^+$ and on $\Gamma_{r,R}^-$,
\item $u$ is n.o.\@ on $\Gamma_{r,R}^+$ and p.o.\@ on $\Gamma_{r,R}^-$,
\item $u$ is n.o.\@ on $\Gamma_{r,R}^+$ and on $\Gamma_{r,R}^-$.
\end{enumerate}
To accommodate for these four cases, we define a polar representation for $u_{\e}$ on $A_{r,R}(x_0)$ whose phase depends on the orientation with respect to $g$ when near the boundary.  Let $\gamma(x)$ be such that $g(x)=e^{i\gamma(x)}$ along $\Gamma_R(x_0)$ with $\gamma_0=\gamma(x_0)$ provided $x_0\in\Gamma$. By modifying the single-valued function $\phi$ from \eqref{brphase} if necessary, the phase $\psi$ for $u_{\e}$ can be given as
\begin{align}\label{eq:upolar}
\psi(\rho,\theta)=\left\{
\begin{array}{ll}
d\theta+\phi(\rho,\theta) & \mbox{if $B_R(x_0)\subset\Omega$}\\[0.5em]
D\theta+\gamma_0+\phi(\rho,\theta) & \mbox{if $x_0\in\Gamma$ and $u$ is p.o.\@ on $\Gamma_{r,R}^+$},\\[0.5em]
D\theta+\gamma_0+\phi(\rho,\theta) \pm\pi & \mbox{if $x_0\in\Gamma$ and $u$ is n.o.\@ on $\Gamma_{r,R}^+$}.
\end{array}
\right.
\end{align}
In this form, $\phi$ is a smooth, single-valued function defined on $A_{r,R}(x_0)$ and can be thought of strictly as a function of $\rho>0$ on $\Gamma_{r,R}^{\pm}$ by the choice of coordinates given in \eqref{eq:paramsgammas}. That is, $\phi=\phi(\rho,\theta(\rho))$ on $\Gamma_{r,R}^{\pm}$. The integers $d=\deg(u;\partial B_{\rho}(x_0))$, $D=\ind(u;\partial B_{\rho}(x_0)\cap\Omega)\in\mathbb{Z}$ are the associated degree and boundary index for $u_{\e}$ respectively. Through representation \eqref{eq:upolar}, the boundary index $D$ determines the orientation of $u$ along $\Gamma_{r,R}^-$. Indeed,  when $R$ is taken to be small and $D$ is even, the phase difference across $\Gamma_{r,R}^{\pm}$ will be approximately an even multiple of $\pi$. In this case, the orientation of $u_{\e}$ with respect to $g$ will be maintained along $\Gamma_{r,R}^{\pm}$ (cases (b) and (d)). When $D$ is odd, the orientation of $u_{\e}$ with respect to $g$ changes sign, giving cases (a) and (c).\\

The function $\phi$ plays an important role in estimating the energy contribution of a defect and it is critical to show that it is appropriately bounded. The following proposition is needed for this estimation process.
\begin{prop}\label{ineq:phibound}
Let $\phi$ be as defined in (\ref{eq:upolar}) and suppose $|u|\geq 1/2$, $|\langle u,g^{\perp}\rangle|\leq 1/4$ on $\Gamma_{r,R}^{\pm}$. Then there exists a constant $C>0$ for which $|\phi|\leq C(|\langle u,g^{\perp}\rangle|+\rho)$. In the special case that $\langle u,g^{\perp}\rangle=0$ on $\Gamma_{r,R}^{\pm}$, we have the simplified bound $|\phi|\leq C\rho$.
\end{prop}
The result of Proposition \ref{ineq:phibound} is claimed in \cite{moser2003} for $g^{\perp}=n$ in the weak tangential case, but is not explicitly shown. We prove it here for completeness.
\begin{proof}
Observe the inner product
\begin{equation*}
|\langle u,g^{\perp}\rangle|=|u||\cos(\psi-(\gamma-\pi/2))|=|u||\sin(\psi-\gamma)|=|u||\sin(\psi-\gamma\pm\pi)|.
\end{equation*}
When $|\langle u,g^{\perp}\rangle|\leq 1/4$ and using the bound $|u|\geq 1/2$, we obtain $|\psi-\gamma|\leq \pi/6$ or $|\psi-\gamma\pm\pi|\leq \pi/6$ for all $x\in\Gamma^{\pm}_{r,R}$ depending on orientation. If $\langle u,g^{\perp}\rangle=0$, then we precisely get $|\psi-\gamma|=0$ or $|\psi-\gamma\pm\pi|=0$ which again depends on the orientation of $u$ with respect to $g$. By considering the four possible orientations for $u_{\e}$ (cases (a)--(d)) separately, it can be shown that there is a constant $c>0$ such that $|\phi|\leq \pi/6 + c\rho$ on $\Gamma^{\pm}_{r,R}$ in the weak tangential case and $|\phi|\leq c\rho$ in the strong tangential case. To see this, we analyze case (a) from above and claim the other cases follow similarly. When $u_{\e}$ is positively oriented on $\Gamma_{r,R}^+$ and negatively oriented on $\Gamma_{r,R}^-$ there is $\xi\in[-\pi/6,\pi/6]$ such that,
\begin{equation*}
\psi-\gamma=D\theta+\gamma_0-\gamma+\phi=\xi\ \quad\mbox{on}\ \Gamma_{r,R}^+
\end{equation*}
with $D\in 2\mathbb{Z}+1$. The triangle inequality gives
\begin{align*}
|\phi|\leq |\xi|+|D\theta|+|\gamma_0-\gamma|\leq \frac{\pi}{6}+c\rho.
\end{align*}
On $\Gamma_{r,R}^-$, we have $D\theta+\gamma_0-\gamma+\phi=D\pi+\xi$ and a similar estimate yields
\begin{equation*}
|\phi|\leq |\xi|+|D||\pi-\theta|+|\gamma_0-\gamma|\leq \frac{\pi}{6}+c\rho.
\end{equation*}
The strong tangential condition corresponds to the scenario where $\xi=0$ for each of the four cases (a)--(d), and so the estimate above can be reduced to $|\phi|\leq C\rho$ in this case, which finishes the proof for solutions satisfying $\langle u,g^{\perp}\rangle=0$ on $\Gamma_{r,R}^{\pm}$.\\\\
Assume $R$ is chosen small enough such that $c\rho\leq \pi/12$, for example, so that $|\phi|\leq \pi/4$ on $\Gamma^{\pm}_{r,R}$. Returning to the inner product and omitting the cases where we consider $\pm\pi$ in the argument, the reverse triangle inequality gives
\begin{equation*}
|\langle u,g^{\perp}\rangle|=|u||\sin(\psi-\gamma)|\geq \frac{1}{2}|\sin(\phi)||\cos(D\theta+\gamma_0-\gamma)|-\frac{1}{2}|\cos(\phi)||\sin(D\theta+\gamma_0-\gamma)|
\end{equation*}
which holds on $\Gamma^{\pm}_{r,R}$ for any of the four orientation scenarios. Since $\Gamma$ and $\gamma$ are smooth, for $R$ taken small enough it holds that
\begin{equation*}
|\sin(D\theta+\gamma_0-\gamma)|,\ |1-|\cos(D\theta+\gamma_0-\gamma)||\leq C\rho,
\end{equation*}
and so we may assume $|\cos(D\theta+\gamma_0-\gamma)|\geq 1/2$ on $\Gamma^{\pm}_{r,R}$. Thus,
\begin{equation*}
|\langle u,g^{\perp}\rangle|\geq\frac{1}{4}|\sin(\phi)|-\frac{1}{2}C\rho.
\end{equation*}
Finally, since $|\phi|\leq \pi/4$ we have
\begin{equation*}
|\langle u,g^{\perp}\rangle|\geq\frac{1}{8}|\phi|-\frac{1}{2}C\rho
\end{equation*}
which finishes the proof.
\end{proof}
As described in much of the surrounding literature, the energy contribution of a non-trivial interior defect for solutions of the Ginzburg--Landau equations on an annulus $A_{r,R}$ is known to be logarithmic in the ratio $R/r$ and depends on the \emph{square} of the degree $d$ of $u$ around the vortex. A similar result holds for boundary defects with associated boundary index $D$. This result is given in Theorem \ref{ineq:localanulus} below.
\begin{thm}\label{ineq:localanulus}
[Strong Tangential Case] Suppose $x_0\in\overline{\Omega}$ and assume that $1/2\leq |u|\leq 1$ in $A_{r,R}(x_0)$. Additionally, suppose $\langle u,g^{\perp}\rangle=0$ on $\Gamma_{r,R}^{\pm}(x_0)$ and that there is some number $K$ such that
\begin{align*}
&E_{\e}(u)\leq K|\ln\e|+K,\\[0.5em]
&\frac{1}{\e^2}\int_{\omega_{\e^{\gamma}}(x_0)}(1-|u|^2)^2\,dx\leq K,
\end{align*}
where $\e^{\gamma}$ is as in Theorem \ref{thm:eta}. Then there exists a constant $C$ depending only on $\Omega$, $\gamma$ and $K$ such that:
\begin{enumerate}[(i)]
\item If $B_R(x_0)\subset\Omega$, $\e\leq r<R\leq r_0$ and $d=\deg(u;\partial B_r(x_0))\neq 0$,
\begin{equation}\label{ineq:lower1}
\int_{A_{r,R}(x_0)}|\nabla u|^2\,dx \geq 2d^2\pi \ln\left(\frac{R}{r}\right)-C.
\end{equation}
\item If $x_0\in\Gamma$, $\e\leq r<R\leq r_0$ and $D=\ind(u;\partial B_r(x_0)\cap\Omega)\neq 0$,
\begin{equation}\label{ineq:lower2}
\int_{A_{r,R}(x_0)}|\nabla u|^2\,dx \geq D^2\pi\ln\left(\frac{R}{r}\right)-C.
\end{equation}
\end{enumerate}
[Weak Tangential Case] Suppose $x_0\in\overline{\Omega}$ and assume that $1/2\leq |u|\leq 1$ in $A_{r,R}(x_0)$. Additionally, suppose $|\langle u,g^{\perp}\rangle|\leq 1/4$ on $\Gamma_{r,R}^{\pm}$ and that there is some number $K$ such that
\begin{align*}
&E_{\e}^{g,s}(u)\leq K|\ln\e|+K,\\[0.5em]
&\frac{1}{\e^2}\int_{\omega_{\e^{\gamma}}(x_0)}(1-|u|^2)^2\,dx+\frac{1}{\e^s}\int_{\Gamma_{\e^{\gamma}}}\langle u,g^{\perp}\rangle^2\,ds\leq K,
\end{align*}
where $\e^{\gamma}$ is as in Theorem \ref{thm:eta}. Then there exists a constant $C$ depending only on $\Omega$, $\gamma$ and $K$ such that:
\begin{enumerate}[(i)]
\item If $B_R(x_0)\subset\Omega$, $\e\leq r<R\leq r_0$ and $d=\deg(u;\partial B_r(x_0))\neq 0$, then \eqref{ineq:lower1} holds.\\
\item If $x_0\in\Gamma$, $\e^s\leq r<R\leq r_0$ and $D=\ind(u;\partial B_r(x_0)\cap\Omega)\neq 0$, then \eqref{ineq:lower2} holds.
\end{enumerate}
\end{thm}
\begin{proof}
The proof of inequality \eqref{ineq:lower1} is omitted since it follows identically to that of \cite[Proposition 3.4]{struwe1994asymptotic} and \cite[Proposition 3.4']{struwe1995}. For \eqref{ineq:lower2}, we provide a brief sketch to show how the boundary index appears and how the boundary conditions are handled. With this, we assume $x_0\in\Gamma$. Using the polar representation for $u$ on $A_{r,R}(x_0)$,
\begin{align*}
\int_{A_{r,R}}|\nabla u|^2\,dx&=\int_{A_{r,R}}\left(f^2|\nabla \psi|^2+|\nabla f|^2\right) dx\\[0.5em]
&\geq \int_{A_{r,R}}f^2|\nabla D\theta+\nabla \phi|^2\,dx\\[0.5em]
&=\int_{A_{r,R}}\frac{D^2f^2}{\rho^2}\,dx+\int_{A_{r,R}}\frac{2Df^2}{\rho^2}\partial_{\theta}\phi\,dx+\int_{A_{r,R}}f^2|\nabla \phi|^2\,dx\\[0.5em]
&=I_1+I_2+I_3.
\end{align*}
The lower estimates for $I_1$ and $I_3$ primarily follow \cite[Proposition 3.4]{struwe1994asymptotic} and \cite[Proposition 3.4']{struwe1995}, with details regarding the boundary given in \cite[Proposition 5.6]{moser2003} and \cite[Proposition 4.3]{alama2015weak}. Specifically, 
\begin{equation*}
I_1\geq D^2\pi^2\ln\left(\frac{R}{r}\right)-C,\quad I_3\geq \frac{1}{4}\int_{A_{r,R}(x_0)}|\nabla \phi|^2\,dx
\end{equation*}
where $C$ is a constant independent of $\e$. For the integral $I_2$, Proposition \ref{ineq:phibound} implies
\begin{equation*}
|\phi(\rho,\theta_2)-\phi(\rho,\theta_1)|\leq 2C\left(\sum_{x\in\partial\Gamma_{\rho}^{\pm}} |u_{\perp}(\rho,\theta_i(\rho))|+\rho\right)
\end{equation*}
with $u_{\perp}=0$ for strong tangential solutions. Applying bounding methods found in \cite[Proposition 5.6]{moser2003} and \cite[Proposition 3.4]{struwe1994asymptotic},
\begin{align*}
|I_2|&\leq\left|\int_{A_{r,R}}\frac{2D}{\rho^2}\partial_{\theta}\phi\,dx\right|+2\left|\int_{A_{r,R}}\frac{D(1-f^2)}{\rho^2}\partial_{\theta}\phi\,dx\right|\\[0.5em]
&\leq\int_{r}^{R}\frac{2|D||\phi(\rho,\theta_2)-\phi(\rho,\theta_1)|}{\rho}\,d\rho+\frac{1}{4}\int_{A_{r,R}}|\nabla \phi|^2\,dx+C\\[0.5em]
&\leq4|D|C\int_{\Gamma_{r,R}^{\pm}}\frac{|\langle u,g^{\perp}\rangle|}{\rho}\,d\rho+\frac{1}{4}\int_{A_{r,R}}|\nabla \phi|^2\,dx+C'.
\end{align*}
If $u$ is a strong tangential solution, the first integral in the last line above does not appear and so the estimate ends there. If $u$ is a weak tangential solution, the proof of \cite[Proposition 5.6]{moser2003}  can be followed with $|f\cdot\nu|$ replaced by $|\langle u,g^{\perp}\rangle|$ throughout, giving
\begin{equation*}
|I_2|\leq C+\frac{1}{4}\int_{A_{r,R}(x_0)}|\nabla \phi|^2\,dx.
\end{equation*}
The desired lower bound is then estimated by
\begin{equation*}
\int_{A_{r,R}}|\nabla u|^2\,dx\geq I_1-|I_2|+I_3\geq D^2\pi\ln\left(\frac{R}{r}\right)-C.
\end{equation*}
\end{proof}
At this point, we are ready to describe a lower bound for the energy on the sets comprising $\mathcal{S}_{\sigma}$ and $\mathcal{S}_{\sigma}^{g,s}$ as defined in \eqref{def:sigma} and \eqref{def:ssigma} respectively.
\begin{lem}\label{lem:onsigmas}[Strong Tangential Case] Suppose $\e_n$ is the subsequence taken in Proposition \ref{prop:sigmaballs} and let $d_i=\deg(u_{\e_n};\partial B_{\sigma}(p_i))$ and $D_j=\ind(u_{\e_n};\partial B_{\sigma}(q_j)\cap\Omega)$. There exists a constant $C$, independent of $\e_n$ and $\sigma$ such that:
\begin{align*}
E_{\e_n}(u_{\e_n};B_{\sigma}(p_i))\geq \pi|d_i|\ln\left(\frac{\sigma}{\e_n}\right)-C,\quad i=1,\ldots,I,\\[0.5em]
E_{\e_n}(u_{\e_n};B_{\sigma}(q_j))\geq \frac{\pi}{2}|D_j|\ln\left(\frac{\sigma}{\e_n}\right)-C,\quad j=1,\ldots,J.
\end{align*}
[Weak Tangential Case] Suppose $\e_n$ is the subsequence taken in Proposition \ref{prop:sigmaballs} and let $d_i=\deg(u_{\e_n};\partial B_{\sigma}(p_i))$ and $D_j=\ind(u_{\e_n};\partial B_{\sigma^s}(q_j)\cap\Omega)$. There exists a constant $C$, independent of $\e_n$ and $\sigma$ such that:
\begin{align*}
E_{\e_n}^{g,s}(u_{\e_n};B_{\sigma}(p_i))\geq \pi|d_i|\ln\left(\frac{\sigma}{\e_n}\right)-C,\quad i=1,\ldots,I,\\[0.5em]
E_{\e_n}^{g,s}(u_{\e_n};B_{\sigma^s}(q_j))\geq \frac{\pi s}{2}|D_j|\ln\left(\frac{\sigma}{\e_n}\right)-C,\quad j=1,\ldots,J.
\end{align*}
\end{lem}
The proof for this lemma comes from a result developed by Sandier \cite{sandier1998lower} (Jerrard \cite{jerrard1999lower} gives a similar result) which uses techniques involving the logarithmic lower bound as found in Theorem \ref{ineq:localanulus}. The method involves a two-step approach where balls containing subsets of $S_{\e}$ (or $S_{\e}^{g,s}$) are expanded and fused such that the energy on these balls can be estimated from below while preserving the natural scaling by $\e$. A fundamental difference between our work and that of Sandier's are details regarding boundary data. Indeed, Sandier's work assumes Dirichlet boundary conditions and thus one does not obtain boundary vortices in this case. For our problem, boundary vortices are expected and thus some extra care needs to be taken when one performs the ball expansion and fusion argument. We refer the reader to \cite{alama2021boojum} and \cite[Lemma 7.1]{alama2020thin} for a proof on how to modify Sandier's result to accommodate for boundary vortices. In particular, the proof not only removes the assumption of Dirichlet boundary data, but also explains how one can deal with the different radial scalings $\e$ and $\e^s$ of the bad balls. However, it is worth noting that the proof of \cite[Lemma 7.1]{alama2020thin} is done in a global sense due to the way boojums must be dealt with. For our case, thanks to Lemma \ref{lem:addthemup}, the arguments of \cite[Lemma 7.1]{alama2020thin} can be applied to each $\sigma$-ball separately which results in Lemma \ref{lem:onsigmas}. \\[1em]
As a consequence of Lemma \ref{lem:onsigmas} and Proposition \ref{prop:upperbound}, we have
\begin{equation}\label{lb:badsetbound1}
\pi\left(\sum_{i=1}^I|d_i|+\frac{1}{2}\sum_{j=1}^J|D_j|\right)|\ln \e_n|-C\leq E_{\e}(u_{\e_n};\mathcal{S}_{\sigma})\leq \pi s\mathcal{D}|\ln\e|+C
\end{equation}
for strong tangential solutions and 
\begin{equation}\label{lb:badsetbound}
\pi\left(\sum_{i=1}^I|d_i|+\frac{s}{2}\sum_{j=1}^J|D_j|\right)|\ln \e_n|-C\leq E_{\e}^{g,s}(u_{\e_n};\mathcal{S}_{\sigma}^{g,s})\leq \pi s\mathcal{D}|\ln\e|+C
\end{equation}
for weak tangential solutions. Using these estimates, we find that each degree $d_i$ and boundary index $D_j$ are uniformly bounded in $\e$ and therefore can be taken to be constant along a subsequence $\e_n\to 0$. It is also clear from \eqref{lb:badsetbound1} and \eqref{lb:badsetbound} that all $\sigma$-balls constituting $\mathcal{S}_{\sigma}$ and $\mathcal{S}_{\sigma}^{g,s}$ respectively, which satisfy $d_i=D_j=0$ do not contribute substantial energy. Therefore, the associated balls can be seen to belong to the set where $u_{\e_n}$ converges. By relabeling the approximate vortices if necessary, we define
\begin{equation*}
\Sigma :=\{p_1,\ldots,p_I\}\cup\{q_1,\ldots,q_J\}
\end{equation*}
to be the collection of all $\sigma$-ball centers with non-trivial degree or boundary index. Upon dividing by $\pi|\ln\e|$ and taking $\e\to 0$ in \eqref{lb:badsetbound1} and \eqref{lb:badsetbound}, it holds that
\begin{equation*}
\sum_{i=1}^I|d_i|+\frac{1}{2}\sum_{j=1}^J|D_j|\leq \mathcal{D}
\end{equation*}
for strong tangential solutions and
\begin{equation*}
\sum_{i=1}^I|d_i|+\frac{s}{2}\sum_{j=1}^J|D_j|\leq s\mathcal{D}
\end{equation*}
for weak tangential solutions. Using identity \eqref{eq:windingidentity} in combination with the above inequalities shows all integers $d_i$ and $D_j$ must be positive (since we've assumed $\mathcal{D}>0$). In fact, we have the equalities
\begin{equation}\label{eq:D}
\mathcal{D}=\sum_{i=1}^Id_i+\frac{1}{2}\sum_{j=1}^JD_j
\end{equation}
and
\begin{equation}\label{sD}
s\mathcal{D}=\sum_{i=1}^Id_i+\frac{s}{2}\sum_{j=1}^JD_j
\end{equation}
for the strong and weak cases respectively. This allows us to conclude the following: Let
\begin{equation*}
\Omega_{\sigma}:=\Omega\setminus\overline{\mathcal{S}_{\sigma}},\quad \Omega_{\sigma}^{g,s}:=\Omega\setminus\overline{\mathcal{S}_{\sigma}^{g,s}}.
\end{equation*} 
\begin{cor}\label{ineq:energysigmabound} [Strong Tangential Case] For any $\sigma\in(0,\sigma_0)$, there exists a constant $C$ independent of $\e$ and $\sigma$ such that
\begin{equation*}
E_{\e_n}(u_{\e_n};\Omega_{\sigma})\leq \pi \mathcal{D}|\ln\sigma|+C.
\end{equation*}
Moreover, there is a constant $C'$ independent of $\e$ such that
\begin{equation*}
\frac{1}{4\e_n^2}\int_{\Omega}\left(1-|u_{\e_n}|^2\right)^2dx\leq C'.
\end{equation*}
[Weak Tangential Case] For any $\sigma\in(0,\sigma_0)$, there exists a constant $C$ independent of $\e$ and $\sigma$ such that
\begin{equation*}
E_{\e_n}^{g,s}(u_{\e_n};\Omega_{\sigma}^{g,s})\leq \pi s\mathcal{D}|\ln\sigma|+C.
\end{equation*}
There is also a constant $C'$ independent of $\e$ such that
\begin{equation*}
\frac{1}{4\e_n^2}\int_{\Omega}\left(1-|u_{\e_n}|^2\right)^2dx+\frac{1}{2\e_n^s}\int_{\Gamma}\langle u_{\e_n},g^{\perp}\rangle^2\,ds\leq C'.
\end{equation*}
\end{cor}
Upon taking an appropriate subsequence $\sigma_n\to 0$, Corollary \ref{ineq:energysigmabound} and following the methods of \cite{bbh} and \cite{struwe1994asymptotic} allows us to conclude that $u_{\e_n}\w u_0$ weakly in $H^1_{loc}(\overline{\Omega}\setminus\Sigma;\mathbb{R}^2)$ as $\e_n\to 0$ where $u_0\in H^1(\overline{\Omega}\setminus\Sigma;\mathbb{S}^1)$ is harmonic. By observing annular regions in $\Omega_{\sigma}$ (and $\Omega_{\sigma}^{g,s}$) and applying Corollary \ref{ineq:energysigmabound} and Theorem \ref{ineq:localanulus}, it is easy to show that vortices of degree or boundary index larger than $1$ require too much energy, and therefore we conclude $d_i=D_j=1$ for all $1\leq i\leq I$, $1\leq j\leq J$. In light of equation \eqref{eq:D}, making the further assumption that $\mathcal{D}=1$ forces that either $\Sigma=\{p_1\}\subset\Omega$ or $\Sigma=\{q_1,q_2\}\subset\Gamma$ which finishes the proof of Theorem \ref{thm:big1}. Moreover, equation $\eqref{sD}$ can be rewritten
\begin{equation*}
s\mathcal{D}=\sum_{i=1}^Id_{i}+\frac{s}{2}\sum_{j=1}^JD_j=(1-s)\sum_{i=1}^Id_i+s\mathcal{D}
\end{equation*}
which immediately implies $\Sigma\subset\Gamma$ whenever $0<s<1$. The fact that each $D_j=1$ also implies $|\Sigma|=J=2\mathcal{D}$ which then completes the proof of Theorem \ref{thm:big2}.
\section{Defect Locations on a Disc when $g=\tau$}
This final section is dedicated to analyzing a strong tangential anchoring example. The primary point of studying this case is to shed light on the fact that certain domain geometries may exist for which boundary vortices could still be energetically preferable to those in the interior, even when all vortices are given equal scaling. We consider the special case where $g=\tau$, the positively oriented unit tangent vector to $\Gamma$, and take $\Omega=B_1(0)$ to be the unit disc for simplicity. In this scenario, $\mathcal{D}=\deg(\tau,\Gamma)=1$ and so in light of equation \eqref{eq:windingidentity} there are only two possibilities for defect locations, exactly one in the interior or exactly two along the boundary. To investigate this further, we observe a renormalized energy.\\[0.5em]
\underline{One Defect in $\Omega$}\\[0.5em]
Let $p\in\Omega$ denote the interior singularity and assume the limiting harmonic map $u_0=\tau$ on $\Gamma$. Following \cite[Section 6]{alama2015weak} and \cite{riviere1999}, consider the solution $\Phi_p$ to
\begin{align*}
\left\{
\begin{array}{ll}
\Delta \Phi_p=2\pi\delta_{p}(x) & \mbox{in}\ \Omega,\\[0.5em]
\dfrac{\partial\Phi_p}{\partial n}=g\times g_{\tau} & \mbox{on}\ \Gamma,
\end{array}
\right.
\end{align*}
with associated asymptotic energy expansion 
\begin{equation*}
E_{\e}^{g,1}(u_{\e})=\pi|\ln\e|+W(p)+c_{\Omega}+o(1)
\end{equation*}
where
\begin{equation*}
W(p)=\lim_{\rho\to 0}\left(\frac{1}{2}\int_{\Omega_{\rho}}|\nabla \Phi_p|^2\,dx -\pi\ln\left(\frac{1}{\rho}\right)\right)
\end{equation*}
is the renormalized energy and $c_{\Omega}$ is the vortex core energy associated to $p$. It can be shown (see \cite{bbh} for example) that the renormalized energy $W$ has a minimum value of zero at the origin $p=0$.\\[0.5em]
\underline{Two Defects on $\Gamma$}\\[0.5em]
Let $q_1, q_2\in\Gamma$ be the boundary singularities and consider the PDE
\begin{align*}
\left\{
\begin{array}{ll}
\Delta \Phi_q=0 & \mbox{in}\ \Omega,\\[0.5em]
\dfrac{\partial\Phi_q}{\partial n}=g\times \partial_{\tau}g -\pi(\delta_{q_1}(x)+\delta_{q_2}(x))& \mbox{on}\ \Gamma
\end{array}
\right.
\end{align*}
which has solution $\Phi_{q}(x)=\ln|x-q_1|+\ln|x-q_2|$. The energy expansion and renormalized energy are
\begin{align*}
E_{\e}^{g,1}(u_{\e})&=\pi|\ln\e|+W(q_1,q_2)+2c_{\Gamma}+o(1),\\
W(q_1,q_2)&=\lim_{\rho\to 0}\left(\frac{1}{2}\int_{\Omega_{\rho}}|\nabla \Phi_q|^2\,dx -\pi\ln\left(\frac{1}{\rho}\right)\right),
\end{align*}
and $c_{\Gamma}$ represents the vortex core energy associated to each $q_j$. By the identity
\begin{equation*}
\int_{\Omega_{\rho}}|\nabla \Phi_q|^2\,dx =\sum_{j=1}^2\int_{\partial B_{\rho}(q_j)\cap\Omega}\Phi_q\frac{\partial\Phi_q}{\partial n_{q_j}}ds+\int_{\Gamma\setminus(\Gamma_{\rho}(q_1)\cup \Gamma_{\rho}(q_2))}\Phi_q\frac{\partial\Phi_q}{\partial n}ds
\end{equation*}
it can be shown via direct calculation that
\begin{equation*}
W(q_1,q_2)=-\pi\ln|q_1-q_2|
\end{equation*} 
which is minimized whenever $q_1$ and $q_2$ are antipodal. In particular,
\begin{equation*}
\min_{q_1,q_2\in\Gamma}W(q_1,q_2)=-\pi\ln|2q_1|<0=\min_{p\in\Omega}W(p).
\end{equation*}
This calculation suggests that the case where $\Sigma=\{q_1,q_2\}$ gives the energetically preferable singularity allocation. To conclude that this is indeed the case, it must be shown that the core energy associated to a boundary vortex is not too large compared to $c_{\Omega}/2$. Although we do not have a rigorous proof for this, we believe it is possible to show using estimates such as those found in \cite{alama2021boojum}.
\bibliographystyle{alpha}

\begin{thebibliography}{DKMO02}

\bibitem[ABG20]{alama2020thin}
Stan Alama, Lia Bronsard, and Dmitry Golovaty.
\newblock Thin film liquid crystals with oblique anchoring and boojums.
\newblock In {\em Annales de l'Institut Henri Poincar{\'e} C, Analyse non
  lin{\'e}aire}. Elsevier, 2020.

\bibitem[ABGS15]{alama2015weak}
Stan Alama, Lia Bronsard, and Bernardo Galv{\~a}o-Sousa.
\newblock Weak anchoring for a two-dimensional liquid crystal.
\newblock {\em Nonlinear Analysis: Theory, Methods \& Applications},
  119:74--97, 2015.

\bibitem[ABM20]{alama2021boojum}
Stan Alama, Lia Bronsard, and Petru Mironescu.
\newblock Inside the light boojums: a journey to the land of boundary defects.
\newblock {\em Analysis in Theory and Applications}, 36(2):128--160, 2020.

\bibitem[BBH94]{bbh}
Fabrice Bethuel, Ha{\"\i}m Brezis, and Fr{\'e}d{\'e}ric H{\'e}lein.
\newblock {\em Ginzburg-Landau Vortices}, volume~13.
\newblock Springer, 1994.

\bibitem[DKMO02]{desimone2002reduced}
Antonio DeSimone, Robert~V Kohn, Stefan M{\"u}ller, and Felix Otto.
\newblock A reduced theory for thin-film micromagnetics.
\newblock {\em Communications on Pure and Applied Mathematics: A Journal Issued
  by the Courant Institute of Mathematical Sciences}, 55(11):1408--1460, 2002.

\bibitem[GCGJ20]{GCGJ}
Carlos~J. Garc\'{\i}a-Cervera, Tiziana Giorgi, and Sookyung Joo.
\newblock Boundary vortex formation in polarization-modulated orthogonal
  smectic liquid crystals.
\newblock {\em SIAM J. Appl. Math.}, 80(5):2024--2044, 2020.

\bibitem[IK21]{IK}
Radu Ignat and Matthias Kurzke.
\newblock Global {J}acobian and {$\Gamma$}-convergence in a two-dimensional
  {G}inzburg-{L}andau model for boundary vortices.
\newblock {\em J. Funct. Anal.}, 280(8):Paper No. 108928, 66, 2021.

\bibitem[Jer99]{jerrard1999lower}
Robert~L Jerrard.
\newblock Lower bounds for generalized ginzburg--landau functionals.
\newblock {\em SIAM Journal on Mathematical Analysis}, 30(4):721--746, 1999.

\bibitem[Kur06]{kurzke2006boundary}
Matthias Kurzke.
\newblock Boundary vortices in thin magnetic films.
\newblock {\em Calculus of Variations and Partial Differential Equations},
  26(1):1--28, 2006.

\bibitem[Mos03]{moser2003}
Roger Moser.
\newblock {Ginzburg-Landau vortices for thin ferromagnetic films}.
\newblock {\em Applied Mathematics Research eXpress}, 2003(1):1--32, 01 2003.

\bibitem[Riv99]{riviere1999}
Tristan Rivi\`{e}re.
\newblock Asymptotic analysis for the ginzburg-landau equations.
\newblock {\em Bollettino dell'Unione Matematica Italiana}, 2-B(3):537--575, 10
  1999.

\bibitem[San98]{sandier1998lower}
Etienne Sandier.
\newblock Lower bounds for the energy of unit vector fields and applications.
\newblock {\em Journal of functional analysis}, 152(2):379--403, 1998.

\bibitem[Str94]{struwe1994asymptotic}
Michael Struwe.
\newblock On the asymptotic behavior of minimizers of the ginzburg-landau model
  in $2 $ dimensions.
\newblock {\em Differential and Integral Equations}, 7(5-6):1613--1624, 1994.

\bibitem[Str95]{struwe1995}
Michael Struwe.
\newblock Erratum: ``on the asymptotic behavior of minimizers of the
  ginzburg--landau model in $2$ dimensions".
\newblock {\em Differential and Integral Equations}, 8(1):224, 1995.

\bibitem[vB22]{vanbrussel22}
Lee van Brussel.
\newblock {\em Boundary Versus Interior Defects for a Ginzburg--Landau Model
  with Tangential Anchoring Conditions}.
\newblock PhD thesis, McMaster University, Hamilton, ON. Canada, June 2022.

\bibitem[VL83]{volovik1983topological}
G.E. Volovik and O.D. Lavrentovich.
\newblock Topological dynamics of defects: boojums in nematic drops.
\newblock {\em Zh Eksp Teor Fiz}, 85(6):1997--2010, 1983.

\end{thebibliography}

\end{document}